\documentclass[12pt]{amsart}
\pdfoutput=1
\usepackage{amssymb,amscd}
\usepackage{graphicx}
\usepackage{mathrsfs}
\usepackage{dsfont}
\usepackage[left=1.4in,right=1.4in,bottom=1.3in,top=1.3in]{geometry}
\PassOptionsToPackage{hyphens}{url}
\usepackage[all,2cell]{xy}
\UseTwocells

\newcommand{\ObSchS}{{\rm{Ob}}((Sch/S)_{fppf})}

\newcommand{\calc}{\mathcal C}
\newcommand{\Ob}{{\rm{Ob}}}

\newcommand{\Z}{\mathbb Z}

\newcommand{\N}{\mathbb N}

\newcommand{\cal}{\mathcal}

 \theoremstyle{plain}
\newtheorem{theorem}{Theorem}[section]
\newtheorem{corollary}[theorem]{Corollary}
\newtheorem{lemma}[theorem]{Lemma}
\newtheorem{proposition}[theorem]{Proposition}
\newtheorem{definition-proposition}[theorem]{Definition/Proposition}

\newtheorem{definition}[theorem]{Definition}

 \theoremstyle{definition}

\theoremstyle{remark}

\newtheorem{remark}[theorem]{Remark}

\numberwithin{equation}{section}

\newcommand{\proref}[1]{Proposition~\ref{#1}}

\makeatletter
\def\@seccntformat#1{\@ifundefined{#1@cntformat}%
    {\csname the#1\endcsname\quad}
    {\csname #1@cntformat\endcsname}}
\newcommand{\section@cntformat}{\S\thesection.\enspace}
\newcommand{\subsection@cntformat}{\S\thesubsection.\enspace}
\newcommand{\subsubsection@cntformat}{\S\thesubsubsection\enspace}
\makeatletter

\usepackage[dvipsnames, svgnames, x11names]{xcolor}
\definecolor{cite}{rgb}{0.50,0.00,1.00}
\definecolor{url}{rgb}{0.00,0.50,0.75}
\definecolor{link}{rgb}{0.00,0.00,0.50}
\usepackage[colorlinks,linkcolor=RoyalBlue,urlcolor=url,citecolor=magenta,breaklinks]{hyperref}

\usepackage{tikz-cd}
\usepackage{capt-of}
\usepackage{enumitem}

\makeatletter
\@namedef{subjclassname@2020}{\textup{2020} Mathematics Subject Classification}
\makeatother

\begin{document}
\title{On the perfection of algebraic spaces}

\author{Tianwei Liang}

\address{School of Mathematics and Information Sciences, Guangzhou University, Guangzhou 510006, P. R. China.}
\email{(Tianwei Liang) 2015200055@e.gzhu.edu.cn}


\begin{abstract}
This paper is a subsequent paper of \cite{Liang1}. We will continue our research on the subject of perfect algebraic spaces that is developed in \cite{Liang1}. By means of algebraic Frobenius morphisms, we define the perfection of arbitrary algebraic spaces of prime characteristic. There is a natural perfection functors on algebraic spaces. We prove several desired properties of the perfection functor. This extends nearly all previous results of the perfection functor on schemes, including the recent ones developed by Bertapelle et al. in \cite{Bertapellea}. Moreover, our theory extends the previous one developed by Xinwen Zhu in \cite{Zhu1,Zhu}.

The perfection functor rises the notion of perfection of sites, which enables us to restate some of Zhu's theory in \cite{Zhu1}. Then we show that our theory of perfect algebraic spaces in \cite{Liang1} is equivalent to Zhu's theory in \cite{Zhu1} when the base scheme is a perfect field of prime characteristic.
\end{abstract}

\maketitle

\tableofcontents

\section{Introduction}
Let $p$ be a prime number and let $\mathbb{F}_{p}$ be a finite field of order $p$. All rings will be tacitly commutative with identity.

\subsection{Background}
The notions of perfection of rings, schemes, and algebraic spaces is particularly important in algebraic geometry. It naturally gives rise to the so-called perfection functor, which enables us to pass between the usual world and the perfect world. The perfection functor has played a more and more significant role in many areas in algebraic geometry (see, for example, \cite{Bertapelle1,Liu,Boyarchenko}). We first begin by briefly reviewing the constructions of perfect closures of rings and schemes in Greenberg's classical paper \cite{Greenberg}.

Let $A$ be a ring of characteristic $p$. In \cite[\S2]{Greenberg}, the \textit{perfect closure} $A^{1/p^{\infty}}$ of $A$ is given by the direct limit
$$
\lim_{\substack{\longrightarrow \\ \phi}}A,
$$
where each transition map $\phi$ is the Frobenius $A\rightarrow A,a\mapsto a^{p}$. This can be globalized to the case of schemes. Let $(X,\mathscr{O}_{X})$ be an $\mathbb{F}_{p}$-scheme. In \cite[\S6]{Greenberg}, the \textit{perfect closure} $X^{1/p^{\infty}}$ of $X$ is given by $(X,\mathscr{O}_{X}^{1/p^{\infty}})$, where $\mathscr{O}_{X}^{1/p^{\infty}}$ is a sheaf of rings on $X$ given by $U\mapsto\mathscr{O}_{X}(U)^{1/p^{\infty}}$.

Now, we turn to the current usage. In \cite[Definition 3.1.2]{Liu}, $A^{1/p^{\infty}}$ is said to be the \textit{direct perfection} of $A$, and we would like to denote it by $A^{pf}$. Analogously, as in \cite[\S3]{Scholze} noted, $X^{1/p^{\infty}}$ is said to be the \textit{perfection} of $X$. We would like to denote it by $X^{pf}$. What different to the previous construction is that $X^{pf}$ can be written as the inverse limit
$$
\lim_{\substack{\longleftarrow \\ \Phi}}X,
$$
where each transition map $\Phi$ is the absolute Frobenius $X\rightarrow X$. This is analogous to the \textit{inverse perfection} of $A$, which is given by the inverse limit
$$
\lim_{\substack{\longleftarrow \\ \phi}}A,
$$
where each transition map $\phi$ is the Frobenius $A\rightarrow A,a\mapsto a^{p}$.

Next, we come to the perfection of algebraic spaces. Let $k$ be a perfect field of characteristic $p$ and let $X$ be any algebraic space over $k$. In \cite[Corollary A.3]{Zhu}, the \textit{perfection} $X^{p^{-\infty}}$ of $X$ is given by the inverse limit
$$
\lim_{\substack{\longleftarrow \\ \sigma_{X}}}X,
$$
where $\sigma_{X}$ is the Frobenius endomorphism $X\rightarrow X$. However, this construction is restricted to algebraic spaces over a perfect field of some prime characteristic. One has no idea whether the construction can be extended to some more general case, for example, algebraic spaces over arbitrary base schemes. On the other hand, the construction is limited by the definition of the Frobenius endomorphism $\sigma_{X}:X\rightarrow X$, which is a globalization of the Frobenius endomorphism of some $k$-algebra.

\subsection{Main results}
In this paper, we continue our research in \cite{Liang1} by focusing on the perfection of algebraic spaces. We first start by upgrading the theories of direct perfection of rings and perfection of schemes. This enables us to apply some desirable properties of the perfection functor on schemes to our study on the perfection of algebraic spaces.

Let $S$ be some base scheme and let $F$ be an algebraic space over $S$. We assume that $F$ has characteristic $p$ so that the algebraic Frobenius $\Psi:F\rightarrow F$ makes sense (see \cite[\S4]{Liang1}). We show that the perfection $F^{pf}$ of $F$ can be written as the inverse limit
$$
\lim_{\substack{\longleftarrow \\ \Psi}}F.
$$
This construction generalizes the previous one in \cite{Zhu1,Zhu}. Let $AP_{S}^{p}$ be the category of algebraic spaces $F$ with $\textrm{char}(F)=p$ and let $\textrm{Perf}_{S}$ denote the category of perfect algebraic spaces over $S$. Our construction naturally gives rise to the perfection functor on algebraic spaces
$$
\underline{\textrm{Perf}}_{S}:AP_{S}^{p}\longrightarrow\textrm{Perf}_{S}.
$$
The perfection functor $\underline{\textrm{Perf}}_{S}$ enjoys several desirable properties as follows.
\begin{proposition}\label{PP1}
Let $Sch$ denote the category of schemes. Let $X$ be an algebraic space of characteristic $p$ over $S$ with perfection $X^{pf}$. Here is a list of properties of the perfection functor $\underline{\rm{Perf}}_{S}$.
\begin{enumerate}[font=\normalfont]
\item
The perfection functor $\underline{\rm{Perf}}_{S}$ is full.
\item
The perfection functor ${\rm{\underline{Perf}}}_{S}$ is right adjoint to the inclusion functor $i:{\rm{Perf}}_{S}\rightarrow AP_{S}^{\neq0}$.
\item
The perfection functor ${\rm{\underline{Perf}}}_{S}$ commutes with the functor $h:Sch\rightarrow\textit{Psh}(Sch)$.
\item
The perfection functor ${\rm{\underline{Perf}}}_{S}$ is left exact, and thus commutes with fibre products.
\item
The perfection functor ${\rm{\underline{Perf}}}_{S}$ induces an equivalence of \'{e}tale sites $X_{\textrm{\'{e}}t}\cong X^{pf}_{\textrm{\'{e}}t}$, and therefore an equivalence of \'{e}tale topoi $\textit{Sh}(X_{\textrm{\'{e}}t})\cong \textit{Sh}(X^{pf}_{\textrm{\'{e}}t})$.
\end{enumerate}
\end{proposition}
Among these properties, (2) generalizes the results in \cite[(5.3)]{Bertapellea} and \cite[Corollary A.3]{Zhu}. Meanwhile, (5) generalizes the results in \cite[Proposition A.5]{Zhu} and \cite[Theorem 3.7]{Scholze}.

The perfection functor $\underline{\textrm{Perf}}_{S}$ enables us to pass the properties between the usual and the perfect world. By means of the properties of $\underline{\textrm{Perf}}_{S}$ in Proposition \ref{PP1}, we establish the following results which completely generalize and improve the previous results in \cite[Lemma 3.4]{Scholze} and \cite[Lemma A.7]{Zhu}.
\begin{theorem}\label{TT1}
Let $f:X\rightarrow Y$ be a morphism of algebraic spaces over $S$ in characteristics $p,q$. Let $f^{\natural}:X^{pf}\rightarrow Y^{pf}$ be the image of $f$ under the perfection functor $\underline{\rm{Perf}}_{S}$. Then
\begin{enumerate}[font=\normalfont]
  \item
$f$ is representable, then $f^{\natural}$ is perfect;
  \item
$f$ has property $\cal{P}$, then $f^{\natural}$ has property $\cal{P}$;
\item
$f$ has property $\cal{P}'$ if and only if $f^{\natural}$ has property $\cal{P}'$;
  \item
$X$ has property $\cal{P}$, then $X^{pf}$ has property $\cal{P}$.
\end{enumerate}
\end{theorem}
\begin{remark}
Here $\cal{P},\cal{P}'$ are properties of schemes or morphisms of schemes satisfying some extra conditions.
\end{remark}

By means of Theorem \ref{TT1}, we discover some nice properties of the algebraic Frobenius morphisms which are analogous to the absolute Frobenius of schemes. Here we list out the properties of the algebraic Frobenius.
\begin{proposition}
Let $F$ be an algebraic space in characteristic $p$ over $S$ with algebraic Frobenius $\Psi_{F}:F\rightarrow F$. Then
\begin{enumerate}[font=\normalfont]
  \item
$\Psi_{F}$ is integral, and is a universal homeomorphism,
  \item
$\Psi_{F}$ is surjective,
  \item
If $F$ is perfect, then $\Psi_{F}$ is weakly-perfect.
\end{enumerate}
\end{proposition}
\begin{remark}
Note that (2) and (3) have been proved in \cite{Liang1}.
\end{remark}

We observe that certain kinds of perfect morphisms can form some topologies, called the \textit{perfect topologies}. Equipped with perfect topologies, one can form some sites $(Spaces/S)_{wp}$, $(Spaces/S)_{qp}$, and $(Spaces/S)_{sp}$. These sites have the same categories of sheaves.
\begin{proposition}
For sites $(Spaces/S)_{wp}$, $(Spaces/S)_{qp}$, and $(Spaces/S)_{sp}$, there is a string of equivalences of topoi
$$
Sh((Spaces/S)_{wp})=Sh((Spaces/S)_{qp})=Sh((Spaces/S)_{sp}).
$$
\end{proposition}
Moreover, we introduce the notion of perfection of sites. This enables us to restate some definitions concerning perfect schemes and perfect algebraic spaces given in \cite{Zhu1}. Next, we want to establish sufficient and necessary condition when a perfect algebraic space is perfect in the sense of \cite{Zhu1}. Surprisingly, our perfect algebraic spaces reduce to Zhu's perfect algebraic spaces when the base scheme is a perfect field $k$ of characteristic $p$.
\begin{theorem}
Let $X$ be an algebraic space over $k$. Then $X$ is perfect in the sense of Zhu if and only if $X$ is perfect.
\end{theorem}
Such a theorem shows that our definitions in \cite[Definition 3.1]{Liang1} all generalize perfect algebraic spaces in \cite{Zhu1,Zhu}. And it indicates that when the base scheme is some particular field $k$, we can apply the results on perfect algebraic spaces in \cite{Zhu1,Zhu} to our perfect algebraic spaces. Therefore, we can study Zhu's perfect algebraic spaces in terms of our perfect algebraic spaces, and we could apply our developed results to Zhu's perfect algebraic spaces.

Let $\textrm{QPerf}_{S}$ be the category of quasi-perfect algebraic spaces over $S$. Let $\textrm{SPerf}_{S}$ be the category of semiperfect algebraic spaces over $S$. Let $\textrm{StPerf}_{S}$ be the category of strongly perfect algebraic spaces over $S$. We have the following string of full embeddings
\begin{align}
\textrm{AlgSp}_{k}^{pf}=\textrm{Perf}_{k}\subset\textrm{StPerf}_{k}\subset\textrm{QPerf}_{k}\subset\textrm{SPerf}_{k}.
\end{align}

\subsection{Outline}
The paper is organized as follows. We begin in \S\ref{S1} by introducing the basic notion of perfect closures of rings and record some important results that will be used in the following sections. Next, in \S\ref{S2}, we collect some fundamental notions surrounding perfection of schemes and deduce some significant results on the perfection functor. Moreover, we relate perfect schemes to equivalence relations. All previous material is put to use in \S\ref{S3}, where we begin by introducing the basic notions of perfection of algebraic spaces. In Proposition \ref{P7}, we construct the perfection of any algebraic spaces in characteristic $p$ using inverse limit through algebraic Frobenius. This naturally yields the perfection functor, which enables us to pass between the usual world and the perfect world. Then we devote the rest of this section to the nature of the perfection functor.

In \S\ref{S4}, we achieve our observation by defining the perfect topologies. Then we prove several desirable properties of the perfect topologies. The perfection functor leads to the notion of perfection of sites. We show that the perfection of some sites inherits certain properties of its original sites.

In \S\ref{S5}, we compare our theory of perfect algebraic spaces with Zhu's theory in \cite{Zhu1,Zhu}. We specify the equivalence between our perfect algebraic spaces and Zhu's perfect algebraic spaces. Moreover, in terms of our developed results, we deduce some properties of Zhu's perfect algebraic spaces that do not seem to appear in \cite{Zhu1,Zhu}. Finally, in \S\ref{B7}, we briefly study properties of groupoids in algebraic spaces under the perfection functor.
\subsection{Conventions}
Throughout this paper, the set of natural numbers will be $\N=\{0,1,2,...,n,...\}$. We will make use of some terminologies and notations in \cite{Stack Project} as follows.
\begin{itemize}
  \item We will denote by $Sch$ the big category of schemes. $Sch_{fppf}$ will be the big fppf site (see \cite[Tag021R]{Stack Project}). And the notation $Sets$ will indicate the big category of sets.
  \item $S$ will always be a base scheme contained in some big fppf site $Sch_{\textit{fppf}}$. Then $(Sch/S)_{\textit{fppf}}$ will denote the big fppf site of $S$ (see \cite[Tag021S]{Stack Project}).
  \item Without explicitly mentioned, all schemes will be contained in $(Sch/S)_{\textit{fppf}}$.
\end{itemize}

We will stick with the definition of algebraic spaces in \cite[Tag025Y]{Stack Project}. By an \textit{algebraic space} over $S$, we mean a fppf sheaf $F$ on $(Sch/S)_{fppf}$ satisfying the usual axioms: its diagonal is representable; and it admits an \'{e}tale cover $h_{U}\rightarrow F$ for a scheme $U\in\textrm{Ob}((Sch/S)_{fppf})$.

Let $AP_{S}$ be the category of algebraic spaces over $S$. Sometimes, we will make use of the Yoneda embedding $(Sch/S)_{fppf}\rightarrow AP_{S},\ T/S\mapsto h_{T}$ and not distinguish between a scheme $T/S$ and the representable algebraic space $h_{T}$.

\section{Perfection of rings}\label{S1}
In this section, we start by reviewing some basic notions of perfect rings and perfect closures of rings. The basic references are \cite{Bertapellea,Bhatt,Greenberg}. Then we deduce some important results concerning the perfect closures.

Let $A$ be a ring of characteristic $p$. Recall that $A$ is said to be \textit{perfect} if the Frobenius endomorphism $F_{A}:A\rightarrow A,\ a\mapsto a^{p}$ is an isomorphism. If $A$ is non-perfect, then one might want to find a perfect ring associated to $A$. This leads to the following definition, which is described using universal property.
\begin{definition}\label{A2}
The {\rm{direct perfection}} of $A$ is a pair $(A_{\rm{perf}},\phi_{A})$ consisting of a perfect ring $A_{\rm{perf}}$ and a ring map $\phi_{A}:A\rightarrow A_{\rm{perf}}$ such that given any such pair $(B,\psi)$, there exists a unique ring map $\psi^{pf}:A_{\rm{perf}}\rightarrow B$ making the diagram
$$
\xymatrix{
  A \ar[rr]^{\phi_{A}} \ar[dr]_{\psi}
                &  &    A_{\textrm{perf}} \ar@{-->}[dl]^{\psi^{pf}}    \\
                & B                 }
$$
commute. The ring map $\phi_{A}:A\rightarrow A_{\rm{perf}}$ is called the canonical inclusion of $A_{\rm{perf}}$.

We say that $A$ is {\rm{$p$-reduced}} if the Frobenius $F_{A}:A\rightarrow A$ is injective, i.e. $a^{p}=0$ implies $a=0$ for all $a\in A$.
\end{definition}

The direct perfection of $A$ can be given by the direct limit
$$
A_{\textrm{perf}}=\lim_{\substack{\longrightarrow \\ n\in\N}}A,
$$
where the transition maps are the Frobenius $F_{A}:A\rightarrow A$ of $A$. One easily reproves the following statement in \cite[\S2]{Greenberg}.
\begin{theorem}
Every ring $A$ of characteristic $p$ has a direct perfection $(A_{\rm{perf}},\phi_{A})$. The canonical inclusion $\phi_{A}:A\rightarrow A_{\rm{perf}}$ is injective if and only if $A$ is $p$-reduced.
\end{theorem}

Clearly, one can consider the dual notion of direct perfection. However, this is not our point in this article.

\begin{definition}
The {\rm{inverse perfection}} of $A$ is a pair $(A^{\rm{perf}},\phi_{A}^{*})$ consisting of a perfect ring $A^{\rm{perf}}$ and a ring map $\phi_{A}^{*}:A^{\rm{perf}}\rightarrow A$ such that given any such pair $(B,\psi)$, there exists a unique ring map $\psi^{pf*}:B\rightarrow A^{\rm{perf}}$ making the diagram
$$
\xymatrix{
  A^{\rm{perf}}  \ar[dr]_{\phi_{A}^{*}}
                &  &    B \ar@{-->}[ll]_{\psi^{pf*}} \ar[dl]^{\psi}    \\
                &  A                }
$$
commute. The ring map $\phi_{A}^{*}:A^{\rm{perf}}\rightarrow A$ is called the canonical projection of $A_{\rm{perf}}$.
\end{definition}

Dually, the inverse perfection of $A$ is given by the inverse limit
$$
A^{\textrm{perf}}=\lim_{\substack{\longleftarrow \\ n\in\N}}A,
$$
where the transition maps are the Frobenius $F_{A}:A\rightarrow A$ of $A$. The inverse perfection of rings is closely related to the tilting functor, see \cite{Scholze1}.

\begin{remark}
In some literature (e.g. \cite[\S1, \S3]{Greenberg}), $A_{\rm{perf}}$ is called the \textit{perfect closure} of $A$, while $A^{\rm{perf}}$ is called the \textit{perfect core} of $A$. Moreover, in this article, we will be only interested in direct perfection. So we will follow \cite{Bertapellea} to denote $A_{\textrm{perf}}$ by $A^{pf}$ and call it the \textit{perfection} of $A$.
\end{remark}

Let $Rings^{p}$ denote the category of rings in characteristic $p$ and let ${\rm{Perf}}R$ denote the category of perfect rings. The perfection of rings naturally gives rise to a functor
$$
Rings^{p}\longrightarrow {\rm{Perf}}R,\ \ A\longmapsto A^{pf},
$$
called the \textit{perfection functor}.

The perfection functor has a right adjoint.
\begin{lemma}
The perfection functor is left adjoint to the inclusion functor ${\rm{Perf}}R\rightarrow Rings^{p}$. In other words, for $A\in\Ob(Rings^{p})$ and $B\in\Ob({\rm{Perf}}R)$, there is a functorial bijection
$$
{\rm{Hom}}_{{\rm{Perf}}R}(R^{pf},S)\xrightarrow{\sim}{\rm{Hom}}_{Rings^{p}}(R,S), \ \ f\longmapsto f\circ \phi_{R}.
$$
\end{lemma}
\begin{proof}
See \cite[(4.2), (4.3)]{Bertapellea} or \cite[Remark 2.2.5]{Williamson}.
\end{proof}

The perfection functor commutes with tensor products.
\begin{lemma}
The perfection functor is right exact. Let $A$ be a ring in characteristic $p$ and let $B,C$ be $A$-algebras. Then we have $(B\otimes_{A}C)^{pf}=B^{pf}\otimes_{A^{pf}}C^{pf}$.
\end{lemma}
\begin{proof}
See the proof of \cite[Lemma 4.10]{Bertapellea}.
\end{proof}

Now, observe the following lemma of perfection of tensor algebras.
\begin{lemma}
Let $R$ be a ring in characteristic $p$ and let $M$ be a $R$-algebra. Then we have an isomorphism of tensor algebras
$$
T(M^{pf})=T(\lim_{\substack{\longrightarrow \\ n\in\N}}M)\cong\lim_{\substack{\longrightarrow \\ n\in\N}}T(M)=T(M)^{pf}.
$$
\end{lemma}
\begin{proof}
This follows from \cite[Tag00DQ]{Stack Project}.
\end{proof}

\section{Perfection of schemes}\label{S2}
In this section, we introduce the notion of perfection of schemes and record some important results. Some material can be found in \cite{Bertapellea,Scholze,Greenberg}. For more details concerning perfect schemes, see \cite[\S2]{Liang1}.

Let $X$ be an $\mathbb{F}_{p}$-scheme. Recall that the \textit{absolute Frobenius} $\Phi_{X}:X\rightarrow X$ of $X$ is the identity on the underlying topological space together with a $p$-th power map on the sheaf of $\mathbb{F}_{p}$-algebras. We say that $X$ is \textit{perfect} if the absolute Frobenius $\Phi_{X}:X\rightarrow X$ is an isomorphism.

If $X$ is a non-perfect $\mathbb{F}_{p}$-scheme, then to find the perfect scheme associated to $X$, we have the following definition in terms of universal property.
\begin{definition}[\cite{Bertapellea}, \S5]
Let $X$ be a scheme of characteristic $p$. The {\rm{perfection}} of $X$ is a pair $(X^{pf},\phi_{X})$ consisting of a perfect scheme $X^{pf}$ of characteristic $p$ and a morphism of schemes $\phi_{X}:X^{pf}\rightarrow X$ such that given any such pair $(Y,\psi)$, there exists a unique morphism of schemes $\psi^{pf}:X^{pf}\rightarrow Y$ making the following diagram commute.
$$
\xymatrix{
  X^{pf}  \ar[dr]_{\phi_{X}}
                &  &    Y \ar@{-->}[ll]_{\psi^{pf}} \ar[dl]^{\psi}    \\
                &  X                }
$$
The morphism $\phi_{X}:X^{pf}\rightarrow X$ is called the canonical projection of $X^{pf}$.
\end{definition}
\begin{remark}
There are different terminologies for $X^{pf}$ in literatures. In \cite[Theorem 8.5.5 (c)]{Liu} and \cite[\S5]{Bertapellea}, $X^{pf}$ is called the \textit{inverse perfection} of $X$. While in \cite[\S7]{Greenberg}, $X^{pf}$ is called the \textit{perfect closure} of $X$. However, we will follow \cite{Milne,Boyarchenko} to call $X^{pf}$ the \textit{perfection} of $X$, since we are only interested in the case of inverse perfection.
\end{remark}

Next, we recall the construction of perfection of schemes in \cite[\S6]{Greenberg} as follows. Let $(X,\mathscr{O}_{X})$ be an $\mathbb{F}_{p}$-scheme. Consider the presheaf $\mathscr{O}_{X}^{pf}$ of perfect rings on $X$ given by
$$
U\longmapsto\mathscr{O}_{X}(U)^{pf}
$$
where $U\subset X$ is open. By \cite[\S5]{Greenberg}, $\mathscr{O}_{X}^{pf}$ is a sheaf. And it follows from \cite[\S7]{Greenberg} that the ringed space $(X,\mathscr{O}_{X}^{pf})$ is a scheme. Therefore, $(X,\mathscr{O}_{X}^{pf})$ is the perfection of $X$. Moreover, the canonical projection $\phi_{X}:X^{pf}\rightarrow X$ is the identity on the underlying topological space. In other words, we have $\left|X^{pf}\right|=\left|X\right|$.

One easily observes that the absolute Frobenius morphism is affine.
\begin{lemma}
Let $X$ be an $\mathbb{F}_{p}$-scheme. The absolute Frobenius $\Phi_{X}:X\rightarrow X$ is affine.
\end{lemma}
\begin{proof}
This is clear.
\end{proof}

Thus, one can describe the perfection of schemes using inverse limit of schemes.
\begin{proposition}[\cite{Scholze}, \S3]\label{P1}
Let $X$ be an $\mathbb{F}_{p}$-scheme. Then we have an isomorphism of schemes
$$
X^{pf}\cong\lim_{\substack{\longleftarrow \\ n\in\N}}X,
$$
where the transition maps are absolute Frobenius $\Phi_{X}:X\rightarrow X$.
\end{proposition}

Observe that the perfection of any $\mathbb{F}_{p}$-scheme is reduced.
\begin{lemma}\label{L5}
Let $X$ be an $\mathbb{F}_{p}$-scheme. If $X$ is perfect, then $X$ is reduced. In particular, the perfection $X^{pf}$ of $X$ is reduced.
\end{lemma}
\begin{proof}
This follows from \cite[V, \S1, no.4, p.A.V.5]{Bourbaki} that every perfect ring is reduced.
\end{proof}

Let $X$ be a scheme over a base scheme $S$. Let $X^{(p)}:=X\times_{S,\Phi_{S}}S_{\Phi_{S}}$ be the base change of $X$ by the absolute Frobenius $\Phi_{S}:S\rightarrow S$, where $S_{\Phi_{S}}$ is regarded as an $S$-scheme via $\Phi_{S}$. Here is the commutative diagram.
$$
\xymatrix{
  X^{(p)} \ar[d]_{} \ar[r]^{\textrm{pr}_{X}} & X \ar[d]^{q} \\
  S \ar[r]^{\Phi_{S}} & S   }
$$

Then another commutative diagram
$$
\xymatrix{
  X \ar[d]_{q} \ar[r]^{\Phi_{X}} & X \ar[d]^{q} \\
  S \ar[r]^{\Phi_{S}} & S   }
$$
(where $\Phi_{X}$ is the absolute Frobenius of $X$) yields a unique an $S$-morphism $\Phi_{X/S}:X\rightarrow X^{(p)}$ such that the diagram
$$
\xymatrix{
  X \ar@/_/[ddr]_{q} \ar@/^/[drr]^{\Phi_{X}}
    \ar@{.>}[dr]|-{\Phi_{X/S}}                   \\
   & X^{(p)} \ar[d]^{} \ar[r]_{\textrm{pr}_{X}}
                      & X \ar[d]_{q}    \\
   & S \ar[r]^{\Phi_{S}}     & S               }
$$
commutes.
\begin{definition}[\cite{SGA3}, VII$_{A}$, \S4]
Consider the situation above. The unique $S$-morphism $\Phi_{X/S}:X\rightarrow X^{(p)}$ is called the relative Frobenius morphism of $X$.
\end{definition}

More generally, for every $n\in\N$, let $X^{(p^{n})}:=X\times_{S,\Phi_{S}^{n}}S_{\Phi_{S}^{n}}$ where $S_{\Phi_{S}^{n}}$ is regarded as an $S$-scheme via $\Phi_{S}^{n}$. Here is the commutative diagram
$$
\xymatrix{
  X^{(p^{n})} \ar[d]_{} \ar[r]^{\textrm{pr}_{X}} & X \ar[d]^{q} \\
  S \ar[r]^{\Phi_{S}^{n}} & S   }
$$
where $\Phi_{S}^{n}$ means composing $\Phi_{S}$ for $n$ times.

If the base scheme $S$ is perfect, then the absolute Frobenius $\Phi_{S}$ is an isomorphism, so we can consider its inverse $\Phi_{S}^{-1}$. Then we can extend $n$ above to any integers, i.e. for every $n\in\Z$, let $X^{(p^{n})}:=X\times_{S,\Phi_{S}^{n}}S_{\Phi_{S}^{n}}$ where $S_{\Phi_{S}^{n}}$ is regarded as an $S$-scheme via $\Phi_{S}^{n}$. Note that we have a canonical isomorphism
$$
(X^{(p^{n})})^{(p)}\cong X^{(p^{n+1})}.
$$

The following lemma provides an alternative description of perfection in a special case.
\begin{lemma}\label{A15}
Let $k$ be a perfect field of characteristic $p$ and let $X$ be a $k$-scheme. Then we have a string of canonical isomorphisms of $k$-schemes:
$$
X^{pf}\cong\lim_{\substack{\longleftarrow \\ n\in\N}} X^{(p^{-n})}\cong\lim_{\substack{\longleftarrow \\ n\in\N}}X,
$$
where the transition maps in the middle are morphisms of $k$-schemes $X^{(p^{-n})}\rightarrow X^{(p^{-n+1})}$ for all $n\in\N$, and the transition maps on the right are absolute Frobenius morphisms $X\rightarrow X$.
\end{lemma}
\begin{proof}
This follows from \cite[Lemma 5.15]{Bertapellea} and Proposition \ref{P1}.
\end{proof}

\section{Perfection of algebraic spaces}\label{S3}
In this section, we extend the subject of perfection of schemes to the setting of algebraic spaces. In order to construct the perfection of a given algebraic space of characteristic $p$, we will make inverse limits of algebraic spaces through algebraic Frobenius. This is different to the previous method in \cite{Zhu}. The construction gives rise to the perfection functor, which transfers every algebraic space of characteristic $p$ to a perfect one.

\subsection{The perfection functor}\

To construct the perfection of an algebraic space, one needs to relate perfect algebraic spaces to perfect equivalence relations of schemes (see \cite[Definition 7.1]{Liang1}).
\begin{theorem}\label{A16}
Let $U\in\ObSchS$ be a perfect scheme over $S$. Assume that $j:R\rightarrow U\times_{S}U$ is an \'{e}tale perfect equivalence relation on $U$ over $S$. Then the quotient sheaf $U/R$ is a perfect algebraic space and $U\rightarrow U/R$ is surjective \'{e}tale such that $(U,R,U\rightarrow U/R)$ is a presentation of $U/R$.
\end{theorem}
\begin{proof}
By \cite[Tag02WW]{Stack Project}, the quotient sheaf $U/R$ is an algebraic space and there is an \'{e}tale cover $U\rightarrow U/R$ such that $R=U\times_{U/R}U$. By assumption, $U$ is perfect such that $U/R$ is perfect.
\end{proof}

Here is the definition of perfections of algebraic spaces. As the usual cases, the perfection of an algebraic space only makes sense when the algebraic space has characteristic $p$.
\begin{definition}\label{A20}
Let $F$ be an algebraic space of characteristic $p$ over $S$. The perfection of $F$ is a pair $(F^{pf},\phi_{F})$ consisting of a perfect algebraic space $F^{pf}$ of characteristic $p$ and a morphism of algebraic spaces $\phi_{F}:F^{pf}\rightarrow F$ such that given any perfect algebraic space $G$ together with a morphism $\psi:G\rightarrow F$, there exists a unique morphism of algebraic spaces $\psi^{pf}:G\rightarrow F^{pf}$ such that the following diagram commute.
$$
\xymatrix{
  F^{pf}  \ar[dr]_{\phi_{F}}
                &  &    G \ar@{-->}[ll]_{\psi^{pf}} \ar[dl]^{\psi}    \\
                &  F                }
$$
\end{definition}

The perfection of an algebraic space is unique up to a unique isomorphism.
\begin{lemma}\label{P8}
Let $F$ be an algebraic space in characteristic $p$ with perfection $F^{pf}$. If there is another perfection $F'^{pf}$ of $F$, then there is a canonical isomorphism $F'^{pf}\cong F^{pf}$.
\end{lemma}
\begin{proof}
This is clear using universal properties of $F^{pf}$ and $F'^{pf}$.
\end{proof}

Next, we will construct the perfection of any algebraic space in characteristic $p$ over $S$. This proves the existence of the perfection of arbitrary algebraic space in characteristic $p$.
\begin{proposition}\label{P7}
Let $F$ be an algebraic space in characteristic $p$ over $S$. Then the perfection of $F$ is the limit
$$
F^{pf}=\lim_{\substack{\longleftarrow \\ n\in\N}}F,
$$
where the transition maps are the algebraic Frobenius morphisms $\Psi_{F}:F\rightarrow F$.
\end{proposition}
\begin{proof}
Let $U\in\ObSchS$ be a scheme of characteristic $p$ and let $h_{U}\rightarrow F$ be a surjective \'{e}tale morphism. Set $R=h_{U}\times_{F}h_{U}$ such that $F=U/R$. Now, Proposition \ref{P1} shows that the limits $U^{pf}=\lim_{i\geq0}U$ and $R^{pf}=\lim_{i\geq0}R$ are perfect schemes.

By the proof of \cite[Tag07SF]{Stack Project}, the morphism $R^{pf}\rightarrow U^{pf}\times_{S}U^{pf}$ is an \'{e}tale equivalence relation. And there is a natural isomorphism
$$
U^{pf}/R^{pf}\longrightarrow\lim_{n\in\N}F
$$
of fppf sheaves on the category of schemes over $S$. Now, by Theorem \ref{A16}, the quotient $U^{pf}/R^{pf}$ is a perfect algebraic space. Thus, the limit $\lim_{n\in\N}F$ is a perfect algebraic space and the universal property of the limit makes it the perfection of $F$.
\end{proof}

The projection $pr_{0}:F^{pf}\rightarrow F$ will be called the \textit{canonical projection} of $F^{pf}$ and denoted by $p_{F}$ or $\phi_{F}$. Next, we prove the functoriality of the perfection of algebraic spaces.
\begin{lemma}\label{A21}
Let $f:F\rightarrow F'$ be a morphism of algebraic spaces in characteristic $p$ over $S$. Let $(F^{pf},\phi_{F})$ and $(F'^{pf},\phi_{F'})$ be the perfections of $F$ and $F'$. Then there exists a canonical map $f^{\natural}:F^{pf}\rightarrow F'^{pf}$ such that the diagram
$$
\xymatrix{
  F^{pf} \ar@{-->}[d]_{f^{\natural}} \ar[r]^{\phi_{F}} & F \ar[d]^{f} \\
  F'^{pf} \ar[r]^{\phi_{F'}} & F'   }
$$
commutes.
\end{lemma}
\begin{proof}
This follows directly from the universal property of the perfection $(F'^{pf},\phi_{F'})$ that $f^{\natural}=(f\circ\phi_{F})^{pf}$.
\end{proof}

Let $AP_{S}^{p}$ denote the category of algebraic spaces in characteristic $p$ over $S$. The perfection of algebraic spaces induces a functor
$$
\textrm{\underline{Perf}}_{S}:AP_{S}^{p}\longrightarrow \textrm{Perf}_{S}, \ \ X\longmapsto X^{pf},
$$
that sends every algebraic space in characteristic $p$ over $S$ to its perfection. And to each morphism $\varphi:X\rightarrow Y$ of algebraic spaces in characteristic $p$ over $S$, it associates the canonical map $\varphi^{\natural}:X^{pf}\rightarrow Y^{pf}$ as in Lemma \ref{A21}. Such a functor $\textrm{\underline{Perf}}_{S}$ is called the \textit{perfection functor}.

Let $i:\textrm{Perf}_{S}\rightarrow AP_{S}^{p}$ be the inclusion functor. The following proposition shows that the perfection functor has a left adjoint.
\begin{proposition}\label{P5}
The perfection functor ${\rm{\underline{Perf}}}_{S}$ is right adjoint to the inclusion functor $i$. In other words, for any $X\in{\rm{Ob}}({\rm{Perf}}_{S})$ and $Y\in{\rm{Ob}}(AP_{S}^{p})$, there exists a functorial bijection
$$
{\rm{Hom}}_{AP_{S}^{p}}(X,Y)\longrightarrow{\rm{Hom}}_{{\rm{Perf}}_{S}}(X,Y^{pf}),\ \ f\longmapsto f^{pf}.
$$
\end{proposition}
\begin{proof}
The inverse function is given by
$$
{\rm{Hom}}_{{\rm{Perf}}_{S}}(X,Y^{pf})\longrightarrow{\rm{Hom}}_{AP_{S}^{p}}(X,Y), \ \ g\longmapsto g\circ\phi_{Y},
$$
where $X\in{\rm{Ob}}({\rm{Perf}}_{S})$ and $Y\in{\rm{Ob}}(AP_{S}^{p})$.

Next, we show that the bijection defined above is functorial. Let $B'\xrightarrow{h}B$ be a morphism in ${\rm{Perf}}_{S}$, let $B\xrightarrow{f}A\xrightarrow{g}A'$ be morphisms in $AP_{S}^{p}$, and let $(A^{pf},\phi_{A}),(A'^{pf},\phi_{A'})$ be the perfections of $A,A'$. By Lemma \ref{A21}, we have $g\phi_{A}=\phi_{A'}g^{\natural}$, which implies that $g\phi_{A}f^{pf}h=\phi_{A'}g^{\natural}f^{pf}h$. Then the universal property of the perfection $A^{pf}$ yields $g\circ f\circ h=\phi_{A'}\circ g^{\natural}\circ f^{pf}\circ h$. Next, by the universal property of $A'^{pf}$, we have $\phi_{A'}\circ (g\circ f\circ h)=\phi_{A'}\circ (g^{\natural}\circ f^{pf}\circ h)$. Now, the uniqueness requirement yields $g\circ f\circ h=g^{\natural}\circ f^{pf}\circ h$. Thus, the bijection is natural in both $X$ and $Y$.
\end{proof}

The perfection functor $\rm{\underline{Perf}}_{S}$ is full but not faithful.
\begin{lemma}\label{L2}
Let $F,F'$ be algebraic spaces in characteristic $p$ over $S$ with perfections $F^{pf},F'^{pf}$. Let $f:F^{pf}\rightarrow F'^{pf}$ be a morphism of algebraic spaces over $S$. Then there exists a morphism $f^{-1}:F\rightarrow F'$ that makes the diagram
$$
\xymatrix{
  F^{pf} \ar[d]_{f} \ar[r]^{\phi_{F}} & F \ar@{-->}[d]^{f^{-1}} \\
  F'^{pf} \ar[r]^{\phi_{F'}} & F'   }
$$
commutes. In other words, the perfection functor $\rm{\underline{Perf}}_{S}$ is full.
\end{lemma}
\begin{proof}
The lemma follows directly from \cite[Lemma 4.5]{Liang1}.
\end{proof}

The perfection functor maps monomorphisms (resp. epimorphisms) to monomorphisms (resp. epimorphisms).
\begin{lemma}\label{LL4}
Let $f:F\rightarrow G$ be a morphism of algebraic spaces in characteristic $p$ over $S$. If $f$ is a monomorphism in the category $AP_{S}$, then $f^{\natural}:F^{pf}\rightarrow G^{pf}$ is a monomorphism in the category ${\rm{Perf}}_{S}$. If $f$ is an epimorphism in the category $AP_{S}$, then $f^{\natural}:F^{pf}\rightarrow G^{pf}$ is an epimorphism in the category ${\rm{Perf}}_{S}$.
\end{lemma}
\begin{remark}
Note that we do not claim that $\rm{\underline{Perf}}_{S}$ preserves epimorphisms as $f^{\natural}$ is not necessarily an epimorphism in the category of algebraic spaces over $S$ if $f$ is an epimorphism.
\end{remark}
\begin{proof}
The first statement follows from Proposition \ref{P5} that the perfection functor has a right adjoint, and thus preserves monomorphisms. For the second statement, let $a,b:G\rightarrow H$ be a morphism in $AP_{S}$ where $H$ is perfect. Assume that we have $af=bf$. Then $a^{\natural}f^{\natural}=b^{\natural}f^{\natural}$ implies that $a^{\natural}=b^{\natural}$. By Lemma \ref{L2}, $a^{\natural},b^{\natural}$ can be arbitrary morphisms in ${\rm{Perf}}_{S}$. This shows that $f^{\natural}$ is an epimorphism in the category ${\rm{Perf}}_{S}$.
\end{proof}

Let $\textit{Psh}(Sch)$ denote the category of presheaves of sets on $Sch$. Next, we show that the perfection functor $\rm{\underline{Perf}}_{S}$ commutes with the functor $h:Sch\rightarrow\textit{Psh}(Sch)$. This requires the following lemma.
\begin{lemma}\label{L1}
Let $\calc$ be a site. Let $F,F',G,G'$ be sheaves of sets on $\calc$. Consider the following commutative solid diagram
$$
\xymatrix{
  F \ar[d]_{c} \ar[r]^{a} & G \ar@{-->}[ld]_{\varphi} \ar[d]^{d} \\
  F' \ar[r]^{b} & G'   }
$$
of morphisms of sheaves of sets. Suppose that $a$ is surjective. Then there exists a unique morphism $\varphi:G\rightarrow F'$ such that the dotted diagram is commutative.
\end{lemma}
\begin{proof}
Let $U\in\Ob(\calc)$. The function $\varphi_{U}:G(U)\rightarrow F'(U)$ is given by
$$
\begin{cases}
\varphi_{U}(a_{U}(x))=c_{U}(x), & \textrm{for all} \ x\in F(U);\\
\varphi_{U}(y)=x_{0},  & \textrm{for all} \ y\in F(X)\setminus \textrm{Im}(a_{U})\textrm{ and some }x_{0}\in F'(X).
\end{cases}
$$
It is clear that the function is well-defined. Let $u:U\rightarrow V$ be a morphism in $\calc$. Then there is a commutative diagram
$$
\xymatrix{
  F(U) \ar[d]_{F(u)} \ar[r]^{a_{U}} & G(U) \ar[d]_{G(u)} \ar[r]^{\varphi_{U}} & F'(U) \ar[d]^{F'(u)} \\
  F(V) \ar[r]^{a_{V}} & G(V) \ar[r]^{\varphi_{V}} & F'(V)   }
$$
Thus, $\varphi$ is a morphism of sheaves. Since $a$ is surjective, this implies that $\varphi$ is unique. Now, we have $bc=b\varphi a=da$. The surjectivity of $a$ implies that $b\varphi=d$.
\end{proof}

Applying the above lemma gives us the desired result.
\begin{lemma}\label{LL1}
Let $U\in\ObSchS$ of characteristic $p$. Then the perfect algebraic space $h_{U^{pf}}$ is the perfection of $h_{U}$, i.e. $h_{U}^{pf}=h_{U^{pf}}$. If $F$ is an algebraic space that is represented by $U$, then $F^{pf}\simeq h_{U^{pf}}$.
\end{lemma}
\begin{proof}
Since $U^{pf}$ is the perfection of $U$, for any perfect scheme $V$, there exists a unique morphism $h_{V}\rightarrow h_{U^{pf}}$ such that the following diagram
$$
\xymatrix{
  h_{U^{pf}}  \ar[dr]_{}
                &  &    h_{V}\ar@{-->}[ll]^{} \ar[dl]^{}    \\
                & h_{U}                 }
$$
commutes. Next, we claim that $h_{U^{pf}}$ satisfies the universal property of perfection. Let $F$ be any perfect algebraic space over $S$ with a morphism $d:F\rightarrow h_{U}$. Let $\varphi_{W}:h_{W}\rightarrow F$ be a surjective \'{e}tale map for $W\in\ObSchS$ a perfect scheme. Consider the following solid diagram
$$
\xymatrix{
  h_{W} \ar[d]_{c} \ar[r]^{\varphi_{W}} & F \ar@{-->}[ld]_{\varphi} \ar[d]^{d} \\
  h_{U^{pf}} \ar[r]^{b} & h_{U}   }
$$
of morphisms of algebraic spaces, where $c$ is the unique morphism induced by $U^{pf}$. By Lemma \ref{L1}, there exists a unique morphism $\varphi:F\rightarrow h_{U^{pf}}$ such that the dotted diagram commutes. Hence, $h_{U^{pf}}$ is the perfection of the algebraic space $h_{U}$.
\end{proof}

In the following, we will study the nature of the perfection functor. The following proposition says that the perfection functor commutes with fibre products.
\begin{proposition}\label{P3}
Let $F,G,H$ be algebraic spaces in characteristic $p$ over $S$. Let $f:F\rightarrow H$ and $g:G\rightarrow H$ be morphisms of algebraic spaces over $S$. Then we have $(F\times_{f,H,g}G)^{pf}=F^{pf}\times_{f^{\natural},H^{pf},g^{\natural}}G^{pf}$. In particular, we have $(F\times G)^{pf}=F^{pf}\times G^{pf}$.
\end{proposition}
\begin{proof}
By Proposition \ref{P5}, the perfection functor $\textrm{\underline{Perf}}_{S}$ has a left adjoint. Since $\textrm{Perf}_{S}$ has finite limits, it follows from \cite[Tag0039]{Stack Project} that $\textrm{\underline{Perf}}_{S}$ is left exact. Thus, $\textrm{\underline{Perf}}_{S}$ commutes with fibre products and finite products by \cite[Tag0035]{Stack Project}.
\end{proof}

We record here two lemmas that will be useful in the sequel.
\begin{lemma}
Let $F$ be an algebraic space of characteristic $p$ over $S$ with perfection $F^{pf}$. Then we have $(F^{pf})^{pf}\cong F^{pf}$.
\end{lemma}
\begin{proof}
It can be shown using universal properties of $F^{pf}$ and $(F^{pf})^{pf}$ that $(F^{pf})^{pf}$ satisfies the same universal property as $F^{pf}$.
\end{proof}

The perfection functor maps representable morphisms to weakly perfect morphisms.
\begin{proposition}\label{P4}
Let $f:F\rightarrow G$ be a representable morphism of algebraic spaces in characteristic $p$ over $S$. Then the morphism $f^{\natural}:F^{pf}\rightarrow G^{pf}$ is weakly perfect.
\end{proposition}
\begin{proof}
Let $U\in\ObSchS$ be a perfect scheme in characteristic $p$ and $\xi\in G(U)$. Choose an isomorphism $h_{W}\simeq h_{U}\times_{\xi,G,f}F$ for some $W\in\ObSchS$ of characteristic $p$. Then by Proposition \ref{P3}, we have $(h_{U}\times_{G}F)^{pf}=h_{U^{pf}}\times_{G^{pf}}F^{pf}\simeq h_{U}\times_{G^{pf}}F^{pf}\simeq h_{W^{pf}}$. Since the perfection functor is full, this shows that $f^{\natural}$ is weakly perfect.
\end{proof}

Moreover, every perfection of an algebraic space is strongly perfect.
\begin{proposition}
Let $F$ be an algebraic space of characteristic $p$ over $S$. Then $F^{pf}$ is strongly perfect.
\end{proposition}
\begin{proof}
Consider the diagonal morphism $\Delta:F\rightarrow F\times F$. Then Proposition \ref{P4} shows that $\Delta^{\natural}:F^{pf}\rightarrow F^{pf}\times F^{pf}$ is weakly perfect. Thus, $F^{pf}$ is strongly perfect.
\end{proof}

In other words, the above proposition says that $F^{pf}$ is the quasiperfection, semiperfection, and strongly perfection of $F$.
\begin{corollary}\label{C1}
Let $F$ be an algebraic space of characteristic $p$ over $S$. Then $F^{pf}$ is perfect, quasi-perfect, semiperfect, and strongly perfect.
\end{corollary}

This surprisingly shows that a perfect algebraic space is quasi-perfect, semiperfect, and strongly perfect.
\begin{lemma}\label{L7}
Let $F$ be an algebraic space of characteristic $p$ over $S$ with perfection $F^{pf}$. If $F$ is perfect, then there is an isomorphism $F\cong F^{pf}$.
\end{lemma}
\begin{proof}
It follows from \cite[Theorem 4.12]{Liang1} that if $F$ is perfect, then the algebraic Frobenius $\Psi_{F}:F\rightarrow F$ is an isomorphism. Thus, the limit $\lim_{\Psi_{F}}F$ is isomorphic to $F$.
\end{proof}

Lemma \ref{L7} immediately yields the following lemma as a corollary.
\begin{lemma}\label{L6}
Let $F$ be a perfect algebraic space over $S$ and let $\varphi_{F}:h_{U}\rightarrow F$ be a surjective \'{e}tale map where $U\in\ObSchS$ is a perfect scheme. Then $\varphi_{F}$ is weakly perfect.
\end{lemma}
\begin{proof}
By Lemma \ref{L7}, every perfect algebraic space is strongly perfect. Thus, it follows from \cite[Theorem 5.5]{Liang1} that the map $\varphi_{F}$ is weakly perfect.
\end{proof}

Meanwhile, we could show that the category of perfect algebraic spaces over $S$ is stable under fibre products. This improves our results in \cite[Proposition 3.5, Proposition 3.6]{Liang1}.
\begin{proposition}\label{PP2}
Let $F\rightarrow H$ and $G\rightarrow H$ be morphisms of algebraic spaces over $S$. If $F,G$ are perfect and $H$ is strongly perfect, then the fibre product $F\times_{H}G$ is a perfect algebraic space. In particular, if $F,G,H$ are perfect, then $F\times_{H}G$ is a perfect algebraic space, and is a fibre product in the category ${\rm{Perf}}_{S}$ of perfect algebraic spaces over $S$.
\end{proposition}
\begin{proof}
Let $h_{U}\rightarrow F,h_{V}\rightarrow G$ be surjective \'{e}tale maps where $U,V\in\ObSchS$ are perfect schemes. Since $H$ is strongly perfect, $h_{U}\times_{H}h_{V}$ is a perfect scheme fulfilling the condition of \cite[Proposition 3.5]{Liang1}. Thus, the fibre product $F\times_{H}G$ is a perfect algebraic space. For the second statement, we have Lemma \ref{L7} which implies that $H$ is strongly perfect.
\end{proof}

Now, we have a string of full embeddings
$$
\textrm{Perf}_{S}\subset \textrm{StPerf}_{S}\subset \textrm{QPerf}_{S}\subset \textrm{SPerf}_{S}\subset AP_{S}.
$$

Here we record a lemma that is useful in the following theorem.
\begin{lemma}\label{L4}
Let $V,X$ be algebraic spaces of characteristic $p$ over $S$. Let $V\rightarrow X$ be a morphism of algebraic spaces over $S$. Then we have an isomorphism $V^{pf}\cong V\times_{X}X^{pf}$. Moreover, the natural projection $p_{X}:X^{pf}\rightarrow X$ is a universal homeomorphism.
\end{lemma}
\begin{proof}
It is easy to see that $V\times_{X}X^{pf}$ satisfies the universal property of perfection $V^{pf}$ using the universal property of base change. For the second statement, choose an \'{e}tale cover $U\rightarrow X$ for $U\in\ObSchS$ of characteristic $p$. Then the base change of $X^{pf}\rightarrow X$ by $U\rightarrow X$ is $U^{pf}\simeq U\times_{X}X^{pf}\rightarrow U$. By \cite[Remark 5.4]{Bertapellea}, the canonical projection $U^{pf}\rightarrow U$ is a universal homeomorphism. Thus, $X^{pf}\rightarrow X$ is a universal homeomorphism.
\end{proof}

Topological invariance of the \'{e}tale site gives the following crucial statement.
\begin{theorem}
Let $X$ be any algebraic space of characteristic $p$ over $S$ with perfection $X^{pf}$. Then the functor given by
$$
(U\rightarrow X)\longmapsto(U^{pf}\simeq U\times_{X}X^{pf}\rightarrow X^{pf})
$$
induces an equivalence of \'{e}tale sites $X_{\textrm{\'{e}}t}\cong X^{pf}_{\textrm{\'{e}}t}$, and therefore an equivalence of \'{e}tale topoi $\widetilde{X_{\textrm{\'{e}}t}}\cong \widetilde{X^{pf}_{\textrm{\'{e}}t}}$.
\end{theorem}
\begin{proof}
By Lemma \ref{L4}, the projection $X^{pf}\rightarrow X$ is a universal homeomorphism. Then it follows from \cite[Tag05ZH]{Stack Project} that this yields an equivalence of sites $X_{\textrm{\'{e}}t}\cong X^{pf}_{\textrm{\'{e}}t}$.
\end{proof}

The perfection functor preserves points of an algebraic space.
\begin{lemma}\label{L3}
Let $X$ be an algebraic space of characteristic $p$ over $S$ with perfection $X^{pf}$. Let $p_{X}:X^{pf}\rightarrow X$ be the canonical projection. Then $p_{X}$ is a monomorphism, is surjective, and is integral. Furthermore, we have a homeomorphism of the underlying topological spaces
$$
\left|p_{X}\right|:\left|X\right|\xrightarrow{\cong}\left|X^{pf}\right|.
$$
\end{lemma}
\begin{proof}
Observe that the diagonal $\Delta_{p_{X}}:X^{pf}\rightarrow X^{pf}\times_{X}X^{pf}\simeq X^{pf}$ is an isomorphism. This shows that $p_{X}$ is a monomorphism. Then we have an injective map $\left|p_{X}\right|:\left|X^{pf}\right|\hookrightarrow\left|X\right|$. Next, choose an \'{e}tale cover $V\rightarrow X$ for $V\in\ObSchS$ of characteristic $p$. Then the composition $V^{pf}\simeq X^{pf}\times_{X}V\rightarrow V$ is surjective and integral as the canonical projection $V^{pf}\rightarrow V$ is a universal homeomorphism. Thus, $p_{X}:X^{pf}\rightarrow X$ is surjective. Now, we have a surjective map $\left|p_{X}\right|:\left|X^{pf}\right|\twoheadrightarrow\left|X\right|$. Therefore, we obtain a bijection $\left|p_{X}\right|:\left|X^{pf}\right|\xrightarrow{\sim}\left|X\right|$ of sets of points, which gives rise to a homeomorphism of underlying topological spaces.
\end{proof}

An algebraic space is representable if and only if its perfection is representable.
\begin{lemma}\label{LL5}
Let $X$ be an algebraic space of characteristic $p$ over $S$ with perfection $X^{pf}$. Then $X$ is a scheme if and only if $X^{pf}$ is a scheme. In particular, $X$ is affine if and only if $X^{pf}$ is affine.
\end{lemma}
\begin{proof}
If $X$ is a scheme, then $X^{pf}$ is a perfect scheme by Lemma \ref{LL1}. Conversely, suppose that $X^{pf}$ is a scheme. Then it follows from Lemma \ref{L3} and \cite[Tag07VV]{Stack Project} that $X$ is a scheme. Next, we prove the second statement. If $X$ is an affine scheme, then clearly $X^{pf}$ is a perfect affine scheme. Now, assume that $X^{pf}$ is an affine scheme. Then $X^{pf}$ is quasi-compact and separated, and so is $X$. Thus, it follows from \cite[Tag07SQ]{Stack Project} that $X$ is also an affine scheme.
\end{proof}

A morphism of algebraic spaces is representable if and only if its perfection is representable.
\begin{proposition}
Let $f:X\rightarrow Y$ be a morphism of algebraic spaces in characteristic $p$ over $S$ and let $f^{\natural}:X^{pf}\rightarrow Y^{pf}$ be its perfection. Then $f$ is representable if and only if $f^{\natural}$ is representable.
\end{proposition}
\begin{proof}
Let $U\in\ObSchS$ of characteristic $p$ with $\xi\in Y(U)$. Note that we have an inclusion $\textrm{Hom}(h_{U},Y^{pf})\subset\textrm{Hom}(h_{U},Y)$ by Yoneda lemma and Lemma \ref{L3}. Now, choose $\xi$ to be some composition $h_{U}\rightarrow Y^{pf}\rightarrow Y$. Then there are isomorphisms $h_{U}\times_{Y}X\simeq h_{U}\times_{Y^{pf}}(X\times_{Y}Y^{pf})\simeq h_{U}\times_{Y^{pf}}X^{pf}\simeq h_{W}$ for some $W\in\ObSchS$. This shows that $f^{\natural}$ is representable.

Conversely, assume that $f^{\natural}$ is representable. Choose $h_{W}\simeq h_{U}\times_{Y^{pf}}X^{pf}$ for some $W\in\ObSchS$ of characteristic $p$ and $\xi'\in Y^{pf}(U)$. Then there is an isomorphism $h_{W^{pf}}\simeq h_{U^{pf}}\times_{Y^{pf}}X^{pf}$. Since the perfection functor is full, it follows from Lemma \ref{LL5} that $h_{U}\times_{Y}X$ is a scheme for every $\xi''\in Y(U)$. Thus, $f$ is representable.
\end{proof}

\subsection{Morphisms of algebraic spaces under the perfection functor}\

In the following, we want to transport properties of algebraic spaces and morphisms from the category $AP_{S}$ to the category $\textrm{Perf}_{S}$. We begin by focusing on properties of representable morphisms of algebraic spaces. The following definition is due to \cite[Tag03HA]{Stack Project}.
\begin{definition}
Let $f:X\rightarrow Y$ be a representable morphism of algebraic spaces in characteristic $p$ over $S$. Let $\cal{P}$ be a property of morphisms of schemes which
\begin{enumerate}[font=\normalfont]
\item
is preserved under arbitrary base change, and
\item
is fppf local on the base.
\end{enumerate}
We say that $f$ has property $\cal{P}$ if for every $U\in\ObSchS$ and any $\xi\in Y(U)$, the morphism of schemes $U\times_{Y}X\rightarrow U$ has property $\cal{P}$.
\end{definition}

The following lemma shows that properties of representable morphisms as above are preserved under the perfection functor.
\begin{lemma}\label{T1}
Let $f:X\rightarrow Y$ be a representable morphism of algebraic spaces in characteristic $p$ over $S$ and let $f^{\natural}:X^{pf}\rightarrow Y^{pf}$ be its perfection. Let $\cal{P}$ be a property of morphisms of schemes which
\begin{enumerate}[font=\normalfont]
\item
is preserved under arbitrary base change, and
\item
is fppf local on the base.
\end{enumerate}
If $f$ has property $\cal{P}$, then $f^{\natural}$ has property $\cal{P}$.
\end{lemma}
\begin{proof}
Assume that $f$ has property $\cal{P}$. Then the morphism of schemes $W\rightarrow U$ induced by $h_{W}\simeq X\times_{Y}h_{U}\rightarrow h_{U}$ for $W,U\in\ObSchS,\xi\in Y(U)$ has property $\cal{P}$. Now, choose $\xi$ to be some composition $h_{U}\rightarrow Y^{pf}\rightarrow Y$. Then there are isomorphisms $h_{U}\times_{Y}X\simeq h_{U}\times_{Y^{pf}}(X\times_{Y}Y^{pf})\simeq h_{U}\times_{Y^{pf}}X^{pf}$. Thus, the morphism $h_{W}\simeq h_{U}\times_{Y^{pf}}X^{pf}\rightarrow h_{U}$ is the same as $W\rightarrow U$ which has property $\cal{P}$. This shows that the perfection $f^{\natural}:X^{pf}\rightarrow Y^{pf}$ has property $\cal{P}$.
\end{proof}

More generally, we can extend Lemma \ref{T1} to not necessarily representable morphisms of algebraic spaces. First, we consider properties of morphisms of schemes which are \'{e}tale local on source-and-target.
\begin{lemma}\label{LL3}
Let $f:X\rightarrow Y$ be a morphism of algebraic spaces in characteristic $p$ over $S$. Let $\cal{P}$ be a property of morphisms of schemes which is \'{e}tale local on source-and-target. Consider commutative diagrams
$$
\xymatrix{
  U \ar[d]_{a} \ar[r]^{h} & V \ar[d]^{b} \\
  X \ar[r]^{f} & Y   }
$$
where $U,V\in\ObSchS$ and $a,b$ are \'{e}tale. The following are equivalent:
\begin{enumerate}[font=\normalfont]
  \item
  For any diagram as above, the morphism of schemes $h$ has property $\cal{P}$.
  \item
  For some diagram as above where $a,b$ is surjective, $U,V$ have characteristic $p$ and the morphism of schemes $h$ has property $\cal{P}$.
\end{enumerate}
\end{lemma}
\begin{proof}
This follows from \cite[Tag03MJ]{Stack Project}.
\end{proof}

Every property of morphisms of schemes which is \'{e}tale local on the source-and-target gives rise to a property of morphisms of algebraic spaces. Here we specialize the definition in \cite[Tag04RD]{Stack Project} to the following case.
\begin{definition}
Let $f:X\rightarrow Y$ be a morphism of algebraic spaces in characteristic $p$ over $S$ and let $\cal{P}$ be a property of morphisms of schemes which is \'{e}tale local on source-and-target. We say that $f$ has property $\cal{P}$ if one of the equivalent conditions in Lemma \ref{LL3} is satisfied.
\end{definition}

Properties of morphisms of algebraic spaces corresponding to properties of morphisms of schemes which are \'{e}tale local on the source-and-target are preserved under the perfection functor.
\begin{lemma}\label{T2}
Let $f:X\rightarrow Y$ be a morphism of algebraic spaces in characteristic $p$ over $S$. Let $\cal{P}$ be a property of morphisms of schemes which
\begin{enumerate}[font=\normalfont]
\item
is \'{e}tale local on source-and-target, and
\item
is preserved under the perfection functor on schemes.
\end{enumerate}
If $f$ has property $\cal{P}$, then $f^{\natural}:X^{pf}\rightarrow Y^{pf}$ has property $\cal{P}$.
\end{lemma}
\begin{proof}
First, assume that $f$ has property $\cal{P}$. Then it follows from Lemma \ref{LL3} that there exists a commutative diagram
$$
\xymatrix{
  U \ar[d]_{} \ar[r]^{g} & V \ar[d]^{} \\
  X \ar[r]^{f} & Y   }
$$
where the vertical arrows are surjective \'{e}tale, $U,V$ have characteristic $p$, and the morphism of schemes $g$ has property $\cal{P}$. By Theorem \ref{T1}, this gives rise to a commutative diagram
$$
\xymatrix{
  U^{pf} \ar[d]_{} \ar[r]^{g} & V^{pf} \ar[d]^{} \\
  X^{pf} \ar[r]^{f} & Y^{pf}   }
$$
that satisfies the same properties as the first commutative diagram by assumption. Thus, $f^{\natural}$ has property $\cal{P}$.
\end{proof}

In the following proposition, we list out some properties of morphisms of algebraic spaces that are preserved under the perfection functor.
\begin{proposition}\label{P2}
Let $f:X\rightarrow Y$ be a morphism of algebraic spaces in characteristic $p$ over $S$ and let $f^{\natural}:X^{pf}\rightarrow Y^{pf}$ be its perfection. Then the following properties hold for $f$ if and only if they hold for $f^{\natural}$:
\begin{enumerate}[font=\normalfont]
\item
surjective,
\item
quasi-compact,
\item
dominant,
\item
(universally) submersive,
\item
(universally) closed,
\item
(universally) open,
\item
a (universal) homeomorphism,
\item\label{It1}
affine,
\item\label{It2}
integral,
\item\label{It3}
(quasi-)separated.
\end{enumerate}
If one of the following holds for $f$, then it also holds for $f^{\natural}$:
\begin{enumerate}[start=11,font=\normalfont]
\item
a monomorphism,
\item
a closed immersion,
\item
an open immersion,
\item
an immersion,
\item
locally separated,
\item\label{It4}
\'{e}tale,
\item\label{It6}
(faithfully) flat,
\item
weakly perfect (quasi-perfect, semiperfect, perfect).
\end{enumerate}
\end{proposition}
\begin{proof}
(1)-(3) all follow directly from Lemma \ref{L3} and definitions. (4)-(7) are by Lemma \ref{L4} that the canonical projection $X^{pf}\rightarrow X$ is a universal homeomorphism.

(8) If $f$ is affine, then $f^{\natural}$ is affine by Lemma \ref{T1}. Conversely, suppose that $f^{\natural}$ is affine. Then for any affine scheme $Z$ with morphism $Z\rightarrow Y^{pf}$, the morphism $Z\times_{Y^{pf}}X^{pf}\rightarrow Z$ is affine. This gives rise to an affine morphism $Z^{pf}\times_{Y^{pf}}X^{pf}\rightarrow Z^{pf}$. By Lemma \ref{LL5}, $Z\times_{Y}X$ is affine for every morphism $Z\rightarrow Y$. This shows that $f$ is affine. The proof of (9) is similar.

(10) Assume that $f$ is separated. Since the diagonal morphism $\Delta_{f}:X\rightarrow X\times_{Y}X$ is representable, it follows from Lemma \ref{T1} that $\Delta_{f}^{\natural}:X^{pf}\rightarrow X^{pf}\times_{Y^{pf}}X^{pf}$ is also a closed immersion. Conversely, if the diagonal $\Delta_{f}^{\natural}:X^{pf}\rightarrow X^{pf}\times_{Y^{pf}}X^{pf}$ is a closed immersion, then it is universally closed. Thus, $\Delta_{f}:X\rightarrow X\times_{Y}X$ is also universally closed by (6). This implies that $f$ is separated by \cite[Tag04Y0]{Stack Project}. The second case is clear by (2).

(11) This is obvious since the diagonal $\Delta_{f}^{\natural}:X^{pf}\rightarrow X^{pf}\times_{Y^{pf}}X^{pf}$ is also an isomorphism. Or see the proof of Lemma \ref{LL4}.

(12)-(14) By definition, every closed immersion (resp. open immersion, resp. immersion) is representable. Hence the statements follow directly from Lemma \ref{T1}.

Clearly, (15) is by (14). (16) follows from Lemma \ref{T2} above. Moreover, (17) is by Lemma \ref{T2} and (1). Finally, (18) is obvious by definitions.
\end{proof}

Next, we show that certain properties of algebraic spaces are stable under the perfection functor. First, we consider properties of schemes which is local in the \'{e}tale topology.
\begin{lemma}\label{LL6}
Let $X$ be an algebraic space in characteristic $p$ over $S$. Let $\cal{P}$ be a property of schemes which is local in the \'{e}tale topology. Then the following are equivalent:
\begin{enumerate}[font=\normalfont]
  \item
there exist some $\mathbb{F}_{p}$-scheme $U$ and an \'{e}tale cover $U\rightarrow X$ such that $U$ has property $\cal{P}$, and
  \item
for every scheme $U$ and any \'{e}tale morphism $U\rightarrow X$, the scheme $U$ has property $\cal{P}$.
\end{enumerate}
\end{lemma}
\begin{proof}
This follows from \cite[Tag03E8]{Stack Project}.
\end{proof}

Every \'{e}tale local property of schemes corresponds to a property of algebraic spaces. Here we specialize the definition in \cite[Tag03E6]{Stack Project} to the following case.
\begin{definition}
Let $X$ be an algebraic space in characteristic $p$ over $S$ and let $\cal{P}$ be a property of schemes which is local in the \'{e}tale topology. We say that $X$ has property $\cal{P}$ if one of the equivalent conditions in Lemma \ref{LL6} is satisfied.
\end{definition}

Properties of algebraic spaces corresponding to \'{e}tale local properties of schemes are preserved under the perfection functor.
\begin{lemma}\label{T3}
Let $X$ be an algebraic space in characteristic $p$ over $S$. Let $\cal{P}$ be a property of schemes
\begin{enumerate}[font=\normalfont]
  \item
which is local in the \'{e}tale topology, and
  \item
which is preserved under the perfection functor on schemes.
\end{enumerate}
If $X$ has property $\cal{P}$, then $X^{pf}$ has $\cal{P}$.
\end{lemma}
\begin{proof}
Choose an \'{e}tale cover $U\rightarrow X$ where $U\in\ObSchS$ has characteristic $p$ and has property $\cal{P}$. This gives rise to another \'{e}tale cover $U^{pf}\rightarrow X^{pf}$ where $U^{pf}$ has property $\cal{P}$ by assumption. Thus, $X^{pf}$ has $\cal{P}$.
\end{proof}

In the following proposition, we list out a few properties of algebraic spaces that are preserved under the perfection functor.
\begin{proposition}
Let $X$ be an algebraic space of characteristic $p$ over $S$ with perfection $X^{pf}$. Then the following properties hold for $X$ if and only if they hold for $X^{pf}$:
\begin{enumerate}[font=\normalfont]
\item
quasi-compact,
\item
quasi-separated,
\item
separated.
\end{enumerate}
If one of the following holds for $X$, then it also holds for $X^{pf}$:
\begin{enumerate}[start=4,font=\normalfont]
\item
locally separated,
\item
reduced.
\end{enumerate}
\end{proposition}
\begin{proof}
(1) is by Lemma \ref{L3}. (2)-(4) follow from Proposition \ref{P2} above. Finally, (5) is by Lemma \ref{T3}.
\end{proof}

As a consequence of Lemma \ref{T3} above, we conclude that every perfect algebraic space is reduced.
\begin{lemma}\label{L9}
Let $X$ be an algebraic space in characteristic $p$ over $S$. If $X$ is perfect, then $X$ is reduced. In particular, the perfection $X^{pf}$ of $X$ is reduced, i.e. $X^{pf}=X_{red}^{pf}$.
\end{lemma}
\begin{proof}
The first statement follows from Lemma \ref{L5} and definitions. The second statement is by Lemma \ref{T3}.
\end{proof}

\begin{corollary}
Let $X$ be a perfect algebraic space over $S$. Let $Y$ be a sheaf on $(Sch/S)_{fppf}$. Let $f:X\rightarrow Y$ be a representable morphism of functors. Suppose that $f$ is surjective \'{e}tale. Then $Y$ is a reduced algebraic space.
\end{corollary}
\begin{proof}
This is by \cite[Proposition 3.12]{Liang1} and Lemma \ref{L9}.
\end{proof}

Now, we can make use of Proposition \ref{P2} to prove the following statement, which says that every algebraic Frobenius is a universal homeomorphism, and is integral.
\begin{proposition}
Let $F$ be an algebraic space of characteristic $p$ over $S$ with algebraic Frobenius morphism $\Psi_{F}$. Then $\Psi_{F}$ is integral, and is a universal homeomorphism.
\end{proposition}
\begin{proof}
Consider the morphism $\Psi_{F}^{\natural}:F^{pf}\rightarrow F^{pf}$, which is an isomorphism due to \cite[Theorem 4.12]{Liang1}. This shows that $\Psi_{F}^{\natural}$ is a universal homeomorphism, and is integral. By Proposition \ref{P2}, $\Psi_{F}$ is a universal homeomorphism, and is integral.
\end{proof}

Next, we can consider subspaces of algebraic spaces. We conclude that the perfection functor preserves all kinds of subspaces of an algebraic space.
\begin{proposition}
Let $F$ be an algebraic space in characteristic $p$ over $S$.
\begin{enumerate}[font=\normalfont]
  \item
If $F'\subset F$ is an open subspace, then $F'^{pf}$ is an open subspace of $F^{pf}$.
  \item
If $F'\subset F$ is a closed subspace, then $F'^{pf}$ is a closed subspace of $F^{pf}$.
  \item
If $F'\subset F$ is a locally closed subspace, then $F'^{pf}$ is a locally closed subspace of $F^{pf}$.
\end{enumerate}
\end{proposition}
\begin{proof}
First, we show that any subspace of $X$ has characteristic $p$. Let $V\in\ObSchS$ in characteristic $p$ and let $V\rightarrow F$ be a surjective \'{e}tale map. Then it follows from \cite[Tag02X1]{Stack Project} that there exists $U\in\ObSchS$ and a commutative diagram
$$
\xymatrix{
  U \ar[d]_{} \ar[r]^{} & V \ar[d]^{} \\
  F' \ar[r]^{} & F   }
$$
This shows that $U$ also has characteristic $p$ so that $F'$ has characteristic $p$. Thus, (1)-(3) all follow from Theorem \ref{T1}.
\end{proof}

\section{Perfection of sites}\label{S4}
In this section, we observe that one can define certain topologies on the category $AP_{S}$ by means of certain perfect morphisms. The perfection functor gives rise to the notion of perfection of sites, which provides us with a framework to restate Zhu's perfect algebraic spaces in \cite{Zhu1}.

\subsection{The perfect sites}\

We first introduce the notion of weakly perfect (resp. quasi-perfect, resp. semiperfect) coverings, which generate the corresponding topology.
\begin{definition}
Let $X$ be an algebraic space over $S$. A weakly perfect (resp. quasi-perfect, resp. semiperfect) covering or simply wp (resp. qp, resp. sp) covering of $X$ is a family $\{f_{i}:X_{i}\rightarrow X\}_{i\in I}$ of morphisms of algebraic spaces over $S$ such that each $f_{i}$ is weakly perfect (resp. quasi-perfect, resp. semiperfect), and
$$
\left|X\right|=\bigcup_{i\in I}\left|f_{i}\right|(\left|X_{i}\right|).
$$
\end{definition}

If the source and all the targets of a weakly perfect (resp. quasi-perfect, resp. semiperfect) covering are representable, then we obtain a weakly perfect (resp. quasi-perfect, resp. semiperfect) covering of schemes.

Clearly, every weakly perfect covering is both a quasi-perfect covering and a semiperfect covering. Meanwhile, every quasi-perfect covering is a semiperfect covering. It can be shown that these notions satisfy the conditions being a topology.
\begin{lemma}
Let $X$ be an algebraic space over $S$.
\begin{enumerate}[font=\normalfont]
  \item
If $X'\rightarrow X$ is an isomorphism, then $\{X'\rightarrow X\}$ is a weakly perfect (resp. quasi-perfect, resp. semiperfect) covering.
  \item
If $\{X_{i}\rightarrow X\}_{i\in I}$ is a weakly perfect (resp. quasi-perfect, resp. semiperfect) covering, and for each $i$, we have a weakly perfect (resp. quasi-perfect, resp. semiperfect) covering $\{X_{ij}\rightarrow X_{j}\}_{j\in J_{i}}$, then $\{X_{ij}\rightarrow X\}_{i\in I,j\in J_{i}}$ is a weakly perfect (resp. quasi-perfect, resp. semiperfect) covering.
  \item
If $\{X_{i}\rightarrow X\}_{i\in I}$ is a weakly perfect (resp. quasi-perfect, resp. semiperfect) covering and $X'\rightarrow X$ is a morphism of algebraic spaces over $S$, then the family $\{X_{i}\times_{X}X'\rightarrow X'\}_{i\in I}$ is a weakly perfect (resp. quasi-perfect, resp. semiperfect) covering.
\end{enumerate}
\end{lemma}
\begin{proof}
(1) is by \cite[Lemma 5.8]{Liang1}. (2) follows from \cite[Proposition 5.2]{Liang1}. And (3) is from \cite[Proposition 5.3]{Liang1}.
\end{proof}

Weakly perfect (resp. quasi-perfect, resp. semiperfect) morphisms are local on the base for the perfect topologies.
\begin{proposition}
Let $\tau\in\{wp,qp,sp\}$. Let $f:X\rightarrow Y$ be a morphism of algebraic spaces over $S$. The property $\cal{P}(f)=``f\textrm{ is weakly perfect''}$ (resp. $\cal{P}(f)=``f\textrm{ is quasi-perfect''}$, resp. $\cal{P}(f)=``f\textrm{ is semiperfect''}$) is $\tau$ local on the base.
\end{proposition}
\begin{proof}
Let $\{Y_{i}\rightarrow Y\}_{i\in I}$ be any $\tau$-covering. If $f$ is weakly perfect (resp. quasi-perfect, resp. semiperfect), then clearly $X\times_{Y}Y_{i}\rightarrow Y_{i}$ is weakly perfect (resp. quasi-perfect, resp. semiperfect). Conversely, let $U\in\ObSchS$ and any $\xi\in Y_{i}(U)$. If $X\times_{Y}Y_{i}\rightarrow Y_{i}$ is weakly perfect (resp. quasi-perfect, resp. semiperfect), then we have $h_{W}\simeq h_{U}\times_{Y_{i}}(Y_{i}\times_{Y}X)\simeq h_{U}\times_{Y}X$ for some perfect scheme $W\in\ObSchS$. Thus, $f$ is weakly perfect (resp. quasi-perfect, resp. semiperfect).
\end{proof}

We record here two propositions of compositions of weakly perfect (resp. quasi-perfect, resp. semiperfect) morphisms that will be useful in the sequel.
\begin{proposition}\label{P6}
Let $f:X\rightarrow Y,g:Y\rightarrow Z$ be morphisms of algebraic spaces over $S$.
\begin{enumerate}[font=\normalfont]
  \item
If $f$ is weakly perfect and $g$ is quasi-perfect (resp. semiperfect), then $g\circ f$ is weakly perfect.
  \item
If $f$ is quasi-perfect and $g$ is weakly perfect (resp. semiperfect), then $g\circ f$ is weakly perfect (resp. quasi-perfect).
  \item
If $f$ is semiperfect and $g$ is weakly perfect (resp. quasi-perfect), then $g\circ f$ is weakly perfect (resp. quasi-perfect).
\item
If $f$ is perfect and $g$ is weakly perfect (resp. quasi-perfect, resp. semiperfect), then $g\circ f$ is perfect. Moreover, if $f$ is weakly perfect (resp. quasi-perfect, resp. semiperfect) and $g$ is perfect, then $g\circ f$ is perfect.
\end{enumerate}
\end{proposition}
\begin{proof}
We just prove (1) since the other are similar. Let $U,V\in\ObSchS$ be perfect schemes and let $\xi\in Y(U),\xi'\in Z(V)$. Choose isomorphisms $h_{W}\simeq h_{U}\times_{Y}X$ and $h_{W'}\simeq h_{V}\times_{Z}Y$ for some perfect schemes $W,W'\in\ObSchS$. Then since $U,V\neq\emptyset$, the composition $h_{U}\times_{Y}X\rightarrow h_{U}\rightarrow h_{V}$ makes sense. And we have $(h_{U}\times_{Y}X)\times_{h_{V}}(h_{V}\times_{Z}Y)\simeq h_{U}\times_{Z}X\simeq h_{W\times_{V}W'}$. Thus, $g\circ f$ is weakly perfect.
\end{proof}

\begin{proposition}\label{P9}
Let $f:Y\rightarrow X,g:Y'\rightarrow X,h:Y\rightarrow Y'$ be morphisms of algebraic spaces over $S$. Suppose that $f=g\circ h$.
\begin{enumerate}[font=\normalfont]
\item
If $f,g$ are (resp. perfect, resp. weakly perfect, resp. quasi-perfect, resp. semiperfect), then $h$ is quasi-perfect.
\item
If $f,h$ are (resp. perfect, resp. weakly perfect, resp. quasi-perfect, resp. semiperfect), then $g$ is quasi-perfect.
\end{enumerate}
\end{proposition}
\begin{proof}
We just prove (1) as the proof of (2) is almost the same. Let $U,V\in\ObSchS$ and let $\xi\in X(U),\xi'\in X(V)$. Choose isomorphisms $h_{W'}\simeq h_{U}\times_{X}Y$ and $h_{W}\simeq h_{V}\times_{X}Y'$ for some perfect schemes $W',W\in\ObSchS$. Since $U,V\neq\emptyset$, the composition $h_{V}\times_{X}Y'\rightarrow h_{V}\rightarrow h_{U}$ makes sense. Then we have $(h_{U}\times_{Y'}Y)\times_{h_{U}}(h_{V}\times_{X}Y')\simeq h_{V}\times_{X}Y\simeq h_{W'}$. But $h_{U}\times_{Y'}Y\times_{h_{U}}h_{W}\simeq h_{W'}$. Thus, we have $h_{W}\times_{Y'}Y\simeq h_{W'}$ such that $h$ is quasi-perfect.
\end{proof}

Quasi-perfect and semiperfect morphisms are local on the source for the perfect topologies.
\begin{proposition}
Let $\tau\in\{wp,qp,sp\}$. Let $f:X\rightarrow Y$ be a morphism of algebraic spaces over $S$. The property $\cal{P}(f)=``f\textrm{ is quasi-perfect''}$ (resp. $\cal{P}(f)=``f\textrm{ is semiperfect''}$) is $\tau$ local on the source.
\end{proposition}
\begin{proof}
Let $\{X_{i}\rightarrow X\}_{i\in I}$ be a weakly perfect covering. If $f$ is quasi-perfect (resp. semiperfect), then each $X_{i}\rightarrow Y$ is quasi-perfect (resp. semiperfect) by Proposition \ref{P6}. Conversely, if each $X_{i}\rightarrow Y$ is quasi-perfect (resp. semiperfect), then it follows from Proposition \ref{P9} that $f$ is quasi-perfect (resp. semiperfect). The rest of the proof are similar.
\end{proof}

In the following, we will define the big perfect sites. First, we consider the cases of schemes.
\begin{definition}\label{D1}
A big weakly perfect site $Sch_{wp}$ is any site given by the category of schemes with weakly perfect coverings.

A big quasi-perfect site $Sch_{qp}$ is any site given by the category of schemes with quasi-perfect coverings.

A big semiperfect site $Sch_{sp}$ is any site given by the category of schemes with semiperfect coverings.
\end{definition}
\begin{remark}
To avoid set-theoretic issues, it suffices to define the sites $Sch_{wp},Sch_{qp},Sch_{sp}$ following the general procedures given in \cite[Tag020M]{Stack Project}.
\end{remark}

Suppose that $S$ is a base scheme contained in $Sch_{wp}$ (resp. $Sch_{qp}$, resp. $Sch_{sp}$). By localization, we obtain the \textit{big weakly perfect} (resp. \textit{quasi-perfect}, resp. \textit{semiperfect}) \textit{site of} $S$, denoted by $(Sch/S)_{wp}$ (resp. $(Sch/S)_{qp}$, resp. $(Sch/S)_{sp}$).

Now, we can more generally define the big perfect sites of algebraic spaces.
\begin{definition}\label{D2}
A big weakly perfect site $(Spaces/S)_{wp}$ is any site given by the category of algebraic spaces over $S$ with weakly perfect coverings.

A big quasi-perfect site $(Spaces/S)_{qp}$ is any site given by the category of algebraic spaces over $S$ with quasi-perfect coverings.

A big semiperfect site $(Spaces/S)_{sp}$ is any site given by the category of algebraic spaces over $S$ with semiperfect coverings.
\end{definition}
\begin{remark}
To avoid set-theoretic issues, it suffices to follow the general procedures given in \cite[Tag03Y6]{Stack Project} to define the sites $(Spaces/S)_{wp},(Spaces/S)_{qp},(Spaces/S)_{sp}$.
\end{remark}

Suppose that $X$ is an algebraic space contained in $(Spaces/S)_{wp}$ (resp. $(Spaces/S)_{qp}$, resp. $(Spaces/S)_{sp}$). By localization at $X$, we obtain the \textit{big weakly perfect} (resp. \textit{quasi-perfect}, resp. \textit{semiperfect}) \textit{site of} $X$, denoted by $(Spaces/X)_{wp}$ (resp. $(Spaces/X)_{qp}$, resp. $(Spaces/X)_{sp}$).

Next, we could define the small perfect sites as follows.
\begin{definition}
Let $X$ be an algebraic space over $S$. The small weakly perfect site $X_{Spaces,wp}$ of $X$ is defined as follows:
\begin{enumerate}[font=\normalfont]
  \item
An object of $X_{Spaces,wp}$ is a morphism $U\rightarrow X$ of algebraic spaces over $S$ which is weakly perfect,
  \item
a morphism $(f:U\rightarrow X)\rightarrow(f':U'\rightarrow X)$ of $X_{Spaces,wp}$ is given by a weakly perfect morphism $\varphi:U\rightarrow U'$ of algebraic spaces over $S$ such that $f=f'\circ\varphi$.
  \item
a family of morphisms $\{f_{i}:(U_{i}\rightarrow X)\rightarrow(U\rightarrow X)\}_{i\in I}$ of $X_{Spaces,wp}$ is a covering if and only if $\left|U\right|=\bigcup_{i\in I}\left|f_{i}\right|(\left|U_{i}\right|)$.
\end{enumerate}
\end{definition}
Similarly, one obtains the \textit{small quasi-perfect site $X_{Spaces,qp}$ of} $X$ and the \textit{small semiperfect site $X_{Spaces,sp}$ of} $X$.

\begin{proposition}
Let $\tau\in\{wp,qp,sp\}$ and let $\tau_{0}\in\{qp,sp\}$. The functor given by
$$
Y_{Spaces,\tau}\longrightarrow X_{Spaces,\tau}, \ \ V\longmapsto V\times_{Y}X
$$
is continuous. The continuous functor $Y_{Spaces,\tau_{0}}\longrightarrow X_{Spaces,\tau_{0}}$ induces a morphism of sites
$$
f_{Spaces,\tau_{0}}:X_{Spaces,\tau_{0}}\longrightarrow Y_{Spaces,\tau_{0}}.
$$
\end{proposition}
\begin{proof}
The first statement is clear. For the second statement, we just need to show that $Y_{Spaces,\tau_{0}}$ has fibre products. This follows from Proposition \ref{P9} and \cite[Proposition 5.3]{Liang1}.
\end{proof}

The following lemma shows that every quasi-perfect (resp. semiperfect) covering admits a weakly perfect refinement.
\begin{lemma}\label{L8}
Let $X$ be an algebraic space over $S$. Let $\{X_{i}\rightarrow X\}_{i\in I}$ be a quasi-perfect (resp. semiperfect) covering. Then there exists a weakly perfect covering that refines $\{X_{i}\rightarrow X\}_{i\in I}$.
\end{lemma}
\begin{proof}
For each $i$, choose a scheme $U_{i}\in\ObSchS$ such that $\varphi_{X_{i}}:U_{i}\rightarrow X_{i}$ is a surjective \'{e}tale map. By Lemma \ref{L6}, each $\varphi_{X_{i}}$ is weakly perfect. Then it follows from Proposition \ref{P6} that $\{U_{i}\rightarrow X\}_{i\in I}$ is a weakly perfect covering that refines $\{X_{i}\rightarrow X\}_{i\in I}$.
\end{proof}

We will see that the previous lemma indeed tells that the sites $(Spaces/S)_{wp}$, $(Spaces/S)_{qp}$, and $(Spaces/S)_{sp}$ have the same categories of sheaves.
\begin{proposition}
For sites $(Spaces/S)_{wp}$, $(Spaces/S)_{qp}$, and $(Spaces/S)_{sp}$, there is a string of equivalences of topoi
$$
\widetilde{(Spaces/S)_{wp}}=\widetilde{(Spaces/S)_{qp}}=\widetilde{(Spaces/S)_{sp}}.
$$
\end{proposition}
\begin{proof}
First, it is obvious that every weakly perfect covering admits a quasi-perfect (resp. semiperfect) refinement. Conversely, it follows directly from Lemma \ref{L8} that every quasi-perfect (resp. semiperfect) covering has a weakly perfect refinement. Thus, \cite[Tag00VX]{Stack Project} yields equalities $Sh((Spaces/S)_{wp})=Sh((Spaces/S)_{qp})=Sh((Spaces/S)_{sp})$ of categories of sheaves.
\end{proof}

Let $k$ be a perfect field of characteristic $p$ such that the affine scheme $\textrm{Spec}(k)$ is over $S$. Let $\textrm{Spec}(k)_{spaces,\textit{\'{e}t}}$ be a small \'{e}tale site of $\textrm{Spec}(k)$, whose objects are \'{e}tale morphisms $U\rightarrow\textrm{Spec}(k)$ of algebraic spaces over $S$. Then it follows from \cite[Proposition 3.9]{Liang1} that every $U\in\Ob(\textrm{Spec}(k)_{spaces,\textit{\'{e}t}})$ is perfect. Moreover, the quasi-perfect topology of $\textrm{Spec}(k)_{spaces,\textit{\'{e}t}}$ are finer than the \'{e}tale topology of $\textrm{Spec}(k)_{spaces,\textit{\'{e}t}}$.
\begin{lemma}
Every \'{e}tale covering in ${\rm{Spec}}(k)_{spaces,\textit{\'{e}t}}$ is a quasi-perfect covering.
\end{lemma}
\begin{proof}
Let $(U\rightarrow\textrm{Spec}(k))$ be an object of the site $\textrm{Spec}(k)_{spaces,\textit{\'{e}t}}$. It follows from \cite[Proposition 5.11]{Liang1} that the morphism $U\rightarrow\textrm{Spec}(k)$ is weakly perfect. Hence by Proposition \ref{P9}, every morphism in $\textrm{Spec}(k)_{spaces,\textit{\'{e}t}}$ is quasi-perfect. This proves our statement.
\end{proof}

\subsection{On the perfection of sites}\

In this subsection, we study the perfection of topologies. We first formalize several types of families of morphisms such that they could admit perfection in the following sense.
\begin{definition}
Let $X$ be an algebraic space over $S$. Let $\cal{U}=\{f_{i}:X_{i}\rightarrow X\}_{i\in I}$ be a family of morphisms of algebraic spaces over $S$ with fixed target $X$.
\begin{enumerate}[font=\normalfont]
  \item
$\cal{U}$ is said to have characteristic $p$ if $X$ and each $X_{i}$ have characteristic $p$.
  \item
$\cal{U}$ is said to have characteristic $0$ if either $X$ or one of $X_{i}$ has characteristic $0$.
  \item
If $\cal{U}$ has characteristic $p$, then the perfection of $\cal{U}$ is defined to be the family $\cal{U}^{pf}=\{f_{i}^{\natural}:X_{i}^{pf}\rightarrow X^{pf}\}_{i\in I}$.
 \item
We say that $\cal{U}$ is perfect if $X$ and each $X_{i}$ are perfect.
\end{enumerate}
We will use ${\rm{char}}(\cal{U})$ to indicate the characteristic of $\cal{U}$.
\end{definition}
\begin{remark}
Note that one should distinguish between perfect covering and weakly perfect (resp. quasi-perfect, resp. semiperfect) coverings.
\end{remark}

The following lemma shows that \'{e}tale (resp. Zariski, resp. weakly perfect, resp. quasi-perfect, resp. semiperfect) topologies are preserved under the perfection functor.
\begin{lemma}\label{L12}
Let $\tau\in\{\textit{\'{e}tale,Zar,wp,qp,sp}\}$. Let $X$ be an algebraic space over $S$. Let $\{f_{i}:X_{i}\rightarrow X\}_{i\in I}$ be a $\tau$-covering of $X$ in characteristic $p$. Then the perfection $\{f_{i}^{\natural}:X_{i}^{pf}\rightarrow X^{pf}\}_{i\in I}$ is a perfect $\tau$-covering of $X^{pf}$.
\end{lemma}
\begin{proof}
By Proposition \ref{P2}, each morphism $f_{i}^{\natural}$ is \'{e}tale (resp. open immersion, resp. weakly perfect, resp. quasi-perfect, resp. semiperfect). Moreover, we have
$$
\left|X^{pf}\right|=\left|X\right|=\bigcup_{i\in I}\left|f_{i}\right|(\left|X_{i}\right|)=\bigcup_{i\in I}\left|f_{i}^{\natural}\right|(\left|X_{i}^{pf}\right|).
$$
Thus, $\{f_{i}^{\natural}:X_{i}^{pf}\rightarrow X^{pf}\}_{i\in I}$ is an \'{e}tale (resp. a Zariski, resp. a weakly perfect, resp. a quasi-perfect, resp. a semiperfect) covering of $X^{pf}$.
\end{proof}

It seems to be natural that one can define the perfection of topologies to be some coverings under the perfection functor. However, we will not make use of this definition. In the following lemma, we observe that certain perfect coverings satisfy the conditions being a topology. This inspires us to make the suitable definition of the perfection of sites.
\begin{lemma}
Let $\tau\in\{\textit{fpqc,fppf,\'{e}tale,Zar,smooth},wp,qp,sp\}$. Let $X$ be a perfect algebraic space over $S$.
\begin{enumerate}[font=\normalfont]
  \item
If $X'\rightarrow X$ is an isomorphism, then $\{X'\rightarrow X\}$ is a perfect $\tau$-covering.
  \item
If $\{X_{i}\rightarrow X\}_{i\in I}$ is a perfect $\tau$-covering, and for each $i$, we have a perfect $\tau$-covering $\{X_{ij}\rightarrow X_{i}\}_{j\in J_{i}}$, then $\{X_{ij}\rightarrow X\}_{i\in I,j\in J_{i}}$ is a perfect $\tau$-covering.
  \item
If $\{X_{i}\rightarrow X\}_{i\in I}$ is a perfect $\tau$-covering and $X'\rightarrow X$ is a morphism of perfect algebraic spaces over $S$, then the family $\{X_{i}\times_{X}X'\rightarrow X'\}_{i\in I}$ is a perfect $\tau$-covering.
\end{enumerate}
\end{lemma}
\begin{proof}
For (1), since $X'\rightarrow X$ is an isomorphism, $X'$ is also perfect. Thus, $\{X'\rightarrow X\}$ is a perfect $\tau$-covering. (2) is clear. For (3), note that the fibre products $X_{i}\times_{X}X'$ are perfect.
\end{proof}

Now, observe that every perfect Zariski covering with affine target can be refined by a perfect standard Zariski convering.
\begin{lemma}
Let $T$ be a perfect affine scheme. Let $\cal{T}=\{T_{i}\rightarrow T\}_{i\in I}$ be a perfect Zariski covering. Then there exists a perfect Zariski covering $\{U_{j}\rightarrow T\}_{j=1,2,...,m}$ which refines $\cal{T}$ such that each $U_{j}$ is a standard open of $T$.
\end{lemma}
\begin{proof}
This follows from \cite[Tag020Q]{Stack Project} and \cite[Proposition 2.5]{Liang1}.
\end{proof}

Meanwhile, we can shows that every perfect fppf (resp. \'{e}tale, resp. fpqc, resp. smooth) covering with affine target can be refined by a perfect standard fppf (resp. \'{e}tale, resp. fpqc, resp. smooth) covering.
\begin{lemma}\label{L14}
Let $\tau\in\{\textit{fppf,\'{e}tale,fpqc,smooth}\}$. Let $T$ be a perfect affine scheme. Let $\cal{T}=\{T_{i}\rightarrow T\}_{i\in I}$ be a perfect $\tau$-covering. Then there exists a perfect $\tau$-covering $\{U_{j}\rightarrow T\}_{j=1,2,...,m}$ which refines $\cal{T}$ such that each $U_{j}$ is an affine open in one of the $T_{i}$.
\end{lemma}
\begin{proof}
This follows from \cite[Tag0218, Tag021P, Tag022E, Tag0222]{Stack Project} and \cite[Proposition 2.5]{Liang1}.
\end{proof}

Furthermore, these also apply to any perfect weakly perfect (resp. quasi-perfect, resp. semiperfect) coverings with affine target.
\begin{lemma}
Let $T$ be a perfect affine scheme. Let $\cal{T}=\{T_{i}\rightarrow T\}_{i\in I}$ be a perfect weakly perfect covering. Then there exists a perfect weakly perfect covering $\{U_{j}\rightarrow T\}_{j\in J}$ which refines $\cal{T}$ such that each $U_{j}$ is an affine open in one of $T_{i}$.
\end{lemma}
\begin{proof}
Assume that $U_{j}$ is an affine open of $T_{i}$. Then the lemma follows from the fact that each $T_{i}$ admits an affine open covering.
\end{proof}

Now, we make the definitions of perfection of sites. The perfection of sites enables us to work in a complete perfect setup in a site. First, we consider perfections of sites of schemes.
\begin{definition}
Let $\tau\in\{\textit{fppf,\'{e}tale,Zar,fpqc,smooth,}wp,qp,sp\}$. Let $Sch_{\tau}$ be a big $\tau$-site containing $S$. Let $(Sch/S)_{\tau}$ and $(\textit{Aff}/S)_{\tau}$ be sites. Then the perfection of $Sch_{\tau}$ is a site $Sch_{\tau}^{pf}$ such that
\begin{enumerate}[font=\normalfont]
  \item
the underlying category of $Sch_{\tau}^{pf}$ is the full subcategory of all perfect schemes in $Sch_{\tau}$,
  \item
the set ${\rm{Cov}}(Sch_{\tau}^{pf})$ of coverings is given by all perfect $\tau$-coverings $\cal{U}$ in $Sch_{\tau}$.
\end{enumerate}
The perfection of $(Sch/S)_{\tau}$ is a site $(Sch/S)_{\tau}^{pf}$ such that
\begin{enumerate}[font=\normalfont]
  \item
the underlying category of $(Sch/S)_{\tau}^{pf}$ is the full subcategory of all perfect schemes in $(Sch/S)_{\tau}$,
  \item
the set ${\rm{Cov}}((Sch/S)_{\tau}^{pf})$ of coverings is given by all perfect $\tau$-coverings $\cal{U}$ in $(Sch/S)_{\tau}$.
\end{enumerate}
The perfection of $(\textit{Aff}/S)_{\tau}$ is a site $(\textit{Aff}/S)_{\tau}^{pf}$ such that
\begin{enumerate}[font=\normalfont]
  \item
the underlying category of $(\textit{Aff}/S)_{\tau}^{pf}$ is the full subcategory of all perfect schemes in $(\textit{Aff}/S)_{\tau}$,
  \item
the set ${\rm{Cov}}((\textit{Aff}/S)_{\tau}^{pf})$ of coverings is given by all perfect standard $\tau$-coverings $\cal{U}$ in $(\textit{Aff}/S)_{\tau}$.
\end{enumerate}
\end{definition}

Next, we consider perfections of sites of algebraic spaces.
\begin{definition}
Let $\tau\in\{\textit{fppf,fpqc,\'{e}tale,Zar,smooth},wp,qp,sp\}$. Let $(Spaces/S)_{\tau}$ be a big $\tau$-site containing an algebraic space $X$ and let $(Spaces/X)_{\tau}$ be a big $\tau$-site of $X$. The perfection of $(Spaces/S)_{\tau}$ is a site $(Spaces/S)^{pf}_{\tau}$ given by
\begin{enumerate}[font=\normalfont]
  \item
the subcategory of all perfect algebraic spaces in $(Spaces/S)_{\tau}$ whose morphisms are weakly perfect, and
  \item
a set ${\rm{Cov}}((Spaces/S)^{pf}_{\tau})$ of coverings which is given by all perfect $\tau$-coverings $\cal{U}$ in $(Spaces/S)_{\tau}$.
\end{enumerate}
The perfection of $(Spaces/X)_{\tau}$ is a site $(Spaces/X)^{pf}_{\tau}$ given by
\begin{enumerate}[font=\normalfont]
  \item
the subcategory of all perfect algebraic spaces in $(Spaces/X)_{\tau}$ whose morphisms are weakly perfect, and
  \item
a set ${\rm{Cov}}((Spaces/X)^{pf}_{\tau})$ of coverings which is given by all perfect $\tau$-coverings $\cal{U}$ in $(Spaces/X)_{\tau}$.
\end{enumerate}
\end{definition}

Observe that the structure morphism of a perfect scheme is weakly perfect.
\begin{lemma}\label{L13}
Let $f:X\rightarrow S$ be a morphism of schemes. If $X$ is perfect, then $f$ is weakly perfect as a morphism of algebraic spaces over $S$.
\end{lemma}
\begin{proof}
It is easy to see using definitions.
\end{proof}

It follows from Lemma \ref{L13} that the morphisms in $Sch_{\tau}^{pf}$ and $(Sch/X)_{\tau}^{pf}$ are automatically weakly perfect. Then we have the following strings of inclusions
\begin{align}
&Sch^{pf}_{\tau}\subset Sch^{pf}_{wp}\subset Sch^{pf}_{qp}\subset Sch^{pf}_{sp}, \\
&(Sch/S)^{pf}_{\tau}\subset (Sch/S)^{pf}_{wp}\subset(Sch/S)^{pf}_{qp}\subset(Sch/S)^{pf}_{sp}, \\
&(\textit{Aff}/S)_{\tau}^{pf}\subset (\textit{Aff}/S)^{pf}_{wp}\subset(\textit{Aff}/S)^{pf}_{qp}\subset(\textit{Aff}/S)^{pf}_{sp}, \\
&(\textit{Spaces}/S)^{pf}_{\tau}\subset(\textit{Spaces}/S)^{pf}_{wp}\subset (\textit{Spaces}/S)^{pf}_{qp}\subset(\textit{Spaces}/S)^{pf}_{sp}, \\
&(\textit{Spaces}/X)^{pf}_{\tau}\subset(\textit{Spaces}/X)^{pf}_{wp}\subset (\textit{Spaces}/X)^{pf}_{qp}\subset(\textit{Spaces}/X)^{pf}_{sp}.
\end{align}

The perfection of sites inherits some properties from the previous sites. The following proposition shows that the perfections of sites have fibre products.
\begin{proposition}
Let $\tau\in\{\textit{fppf,\'{e}tale,Zar,fpqc,smooth},wp,qp,sp\}$. Let $Sch_{\tau}$ be a big $\tau$-site containing $S$. The underlying categories of the sites $Sch_{\tau}, Sch_{\tau}^{pf},(Sch/S)_{\tau}^{pf}$, and $(\textit{Aff}/S)^{pf}_{\tau}$ have fibre products. In each case, the obvious inclusion functor into $Sch$ commutes with fibre products. Suppose that $S$ is perfect. Then the category $(Sch/S)_{\tau}^{pf}$ has a final object, i.e. $S/S$.
\end{proposition}
\begin{proof}
The proof is similar to the usual case, see \cite[Tag021D]{Stack Project} for example.
\end{proof}

In the following proposition, we will show that perfect sites and perfect affine sites have the same categories of sheaves.
\begin{proposition}\label{P11}
Let $\tau\in\{\textit{fpqc,fppf,\'{e}tale,smooth}\}$. Let $Sch_{\tau}$ be a big $\tau$-site containing $S$. Then the inclusion functor $u:(\textit{Aff}/S)^{pf}_{\tau}\rightarrow(\textit{Sch}/S)^{pf}_{\tau}$ is special cocontinuous and induces an equivalence of topoi
$$
\widetilde{(\textit{Aff}/S)^{pf}_{\tau}}\cong \widetilde{(\textit{Sch}/S)^{pf}_{\tau}}.
$$
\end{proposition}
\begin{proof}
We will use \cite[Tag03A0]{Stack Project} to prove this proposition, i.e. one have to verify assumptions (1)-(5) of \cite[Tag03A0]{Stack Project}. First, Lemma \ref{L14} implies that the functor $u$ is cocontinuous, since every perfect $\tau$-covering of $T/S$, $T$ affine, admits a perfect standard $\tau$-refinement. So the assumption (1) holds. Next, it follows from \cite[Proposition 2.5]{Liang1} that every perfect scheme has a perfect affine open covering. This proves the assumption (5). Then the proof of (2)-(4) are similar to the proof of \cite[Tag021V]{Stack Project}. Thus, we obtain the equivalence of topoi as desired.
\end{proof}

\section{Comparison with Zhu's perfect algebraic spaces}\label{S5}
In this section, we compare our perfect algebraic spaces in \cite[Definition 3.1]{Liang1} with Zhu's perfect algebraic spaces in \cite{Zhu1,Zhu}.

Let $k$ be a perfect field of characteristic $p$ and let $Sch/k$ denote the category of $k$-schemes. We endow $Sch/k$ with Zariski and \'{e}tale topologies, and obtain the big Zariski site $(Sch/k)_{\textit{Zar}}$ and the big \'{e}tale site $(Sch/k)_{\textit{\'{e}t}}$ with perfections $(Sch/k)^{pf}_{\textit{Zar}}$, $(Sch/k)^{pf}_{\textit{\'{e}t}}$. We denote by $\textit{Aff}/k$ the category of affine $k$-schemes. And $k$-alg will be the category of $k$-algebras with the full subcategory $k\textrm{-alg}^{pf}$ of perfect $k$-algebras. Note that there is an anti-equivalence $k\textrm{-alg}\simeq\textit{Aff}/k$. Hence, one may identify $(\textit{Aff}/k)^{opp}$ with $k\textrm{-alg}$. Next, we can also equip $\textit{Aff}/k$ with Zariski and \'{e}tale topologies. This gives rise to the big affine Zariski site $(\textit{Aff}/k)_{\textit{Zar}}$ and the big affine \'{e}tale site $(\textit{Aff}/k)_{\textit{\'{e}t}}$. The sites $(\textit{Aff}/k)_{\textit{Zar}}^{pf}$ and $(\textit{Aff}/k)_{\textit{\'{e}t}}^{pf}$ will be the perfections of $(\textit{Aff}/k)_{\textit{Zar}}$ and $(\textit{Aff}/k)_{\textit{\'{e}t}}$ of perfect $k$-affine schemes.

If we denote the underlying category of $(\textit{Aff}/k)_{\textit{Zar}}^{pf}$ or $(\textit{Aff}/k)_{\textit{\'{e}t}}^{pf}$ by $(\textit{Aff}/k)^{pf}$, then we have an anti-equivalence $(\textit{Aff}/k)^{pf}\simeq k\textrm{-alg}^{pf}$.

By a \textit{presheaf}, we mean a contravariant functor from $(\textit{Aff}/k)^{pf}$ to the category of sets. Then we first restate the definitions of perfect affine schemes and perfect schemes in \cite[Definition A.1.2]{Zhu1} as follows.
\begin{definition}
A perfect affine scheme ${\rm{Spec}}(R)$ over $k$ is a presheaf on $(\textit{Aff}/k)^{pf}$ which is of the form ${\rm{Hom}}_{k\textrm{-alg}^{pf}}(R,-)$. A perfect scheme $X$ over $k$ is a sheaf on the site $(\textit{Aff}/k)_{\textit{Zar}}^{pf}$ that admits a Zariski cover by perfect affine schemes.
\end{definition}

Next, we restate the definition of Zhu's perfect algebraic spaces in our terminologies, see \cite[Definition A.1.3]{Zhu1}. Recall that a map of presheaves is \textit{schematic} if and only if it is representable.
\begin{definition}
A perfect algebraic space $X$ in the sense of Zhu over $k$ is a sheaf on the site $(\textit{Aff}/k)_{\textit{\'{e}t}}^{pf}$ such that the diagonal is schematic, and there exists a surjective \'{e}tale map $U\rightarrow X$ from a perfect scheme $U$ over $k$.
\end{definition}

The following proposition gives an alternative definition of perfect algebraic space $X$ in the sense of Zhu over $k$.
\begin{proposition}\label{P12}
Let $X'$ be a perfect algebraic space over $k$ in the sense of Zhu. Then $X'$ extends uniquely to a sheaf $X$ on the site $(Sch/k)^{pf}_{\textit{\'{e}t}}$ such that
\begin{enumerate}[font=\normalfont]
\item
for every pair of schemes $U,V\in\Ob((Sch/k)_{\textrm{\'{e}t}}^{pf})$ and any $a\in X(U),b\in X(V)$, the functor $U\times_{a,X,b}V$ is a scheme $W\in\Ob((Sch/k)_{\textrm{\'{e}t}}^{pf})$.
\item
there exists a surjective \'{e}tale map $U\rightarrow X$ from a scheme $U\in\Ob((Sch/k)_{\textrm{\'{e}t}}^{pf})$.
\end{enumerate}
\end{proposition}
\begin{proof}
This follows from Proposition \ref{P11} and definitions.
\end{proof}

There is another definition of perfect algebraic spaces in \cite[A.1.2]{Zhu} which is easier for us to deal with. Recall that the \textit{Frobenius endomorphism} of an algebraic space over $k$ is given by evaluation at the $p$-th power morphism.
\begin{definition}
An algebraic space $X$ over $k$ is perfect in the sense of Zhu if the Frobenius endomorphism $\sigma_{X}:X\rightarrow X$ is an isomorphism. And the Zhu's perfection of $X$ will be denoted by $X^{p^{-\infty}}:=\lim\limits_{\longleftarrow}$$_{\sigma_{X}}X$.
\end{definition}
We will make use of this definition since it suffices to fulfill our purpose of comparison. Note that the algebraic spaces in fppf topology are the same as the algebraic spaces in \'{e}tale topology due to \cite[Tag076M]{Stack Project}. So we make use of this identification.

In the following theorem, we will show that Zhu's perfect algebraic spaces is equivalent to our perfect algebraic spaces.
\begin{theorem}\label{T4}
Let $X$ be an algebraic space over $k$. Then $X$ is perfect in the sense of Zhu if and only if $X$ is perfect.
\end{theorem}
\begin{proof}
By the proof of \cite[Corollary A.3]{Zhu}, every perfect algebraic space $X$ in the sense of Zhu is perfect and strongly perfect. However, due to Lemma \ref{L7}, every perfect algebraic space is strongly perfect. Thus, $X$ is perfect.

Conversely, suppose that $X$ is perfect. Choose an \'{e}tale cover $U\rightarrow X$ where $U$ is a perfect scheme over $k$. Then there are isomorphisms $U^{p^{-\infty}}\simeq U\times_{X}X^{p^{-\infty}}\simeq U$. Thus, we have an isomorphism $X^{p^{-\infty}}\simeq X$. This shows that $X$ is perfect in the sense of Zhu.
\end{proof}

Let $\textrm{AlgSp}_{k}^{pf}$ denote the category of perfect algebraic spaces over $k$ in the sense of Zhu. Then Theorem \ref{T4} yields the following string of full embeddings
\begin{align}
\textrm{AlgSp}_{k}^{pf}=\textrm{Perf}_{k}\subset\textrm{StPerf}_{k}\subset\textrm{QPerf}_{k}\subset\textrm{SPerf}_{k}.
\end{align}

Moreover, we have the following theorem that gives an alternative definition of perfect algebraic spaces in the sense of Zhu over $k$.
\begin{theorem}
Let $F$ be an algebraic space over $k$. Then $F$ is perfect in the sense of Zhu if and only if the following statements are satisfied:
\begin{enumerate}[font=\normalfont]
  \item
  For all schemes $U\in\Ob((Sch/k)_{\textrm{\'{e}t}}^{pf})$, the maps $U\rightarrow F$ are weakly perfect.
  \item
  There exists a weakly perfect, surjective, and \'{e}tale map $V\rightarrow F$ from a scheme $V\in\Ob((Sch/k)_{\textrm{\'{e}t}}^{pf})$. In other words, there is a covering map $V\rightarrow F$ for the perfection of the \'{e}tale topology.
\end{enumerate}
\end{theorem}
\begin{proof}
This follows from Theorem \ref{T4} and \cite[Theorem 5.5]{Liang1}.
\end{proof}

The following proposition characterizes the algebraic Frobenius morphism of a perfect algebraic space in the sense of Zhu.
\begin{proposition}
Let $F$ be an algebraic space over $k$. If $F$ is perfect in the sense of Zhu, then the algebraic Frobenius morphism $\Psi_{F}:F\rightarrow F$ of $F$ is weakly perfect, i.e. $\Psi_{F}$ is a covering map for the perfection of the weakly perfect topology.
\end{proposition}
\begin{proof}
This follows directly from \cite[Proposition 5.10]{Liang1} due to the equivalence described in Theorem \ref{T4}.
\end{proof}

We can show that the category ${\rm{AlgSp}}_{k}^{pf}$ of perfect algebraic spaces over $k$ in the sense of Zhu is stable under fibre products.
\begin{proposition}
Let $F\rightarrow H$ and $G\rightarrow H$ be morphisms of algebraic spaces over $k$. If $F,G,H$ are perfect in the sense of Zhu, then the fibre product $F\times_{H}G$ is also perfect in the sense of Zhu. It is a fibre product in the category ${\rm{AlgSp}}_{k}^{pf}$ of perfect algebraic spaces over $k$ in the sense of Zhu.
\end{proposition}
\begin{proof}
This is by \proref{PP2} due to the equivalence described in Theorem \ref{T4}.
\end{proof}

\section{Perfection of groupoids in algebraic spaces}\label{B7}
In this section, we study properties of group algebraic spaces and groupoids in algebraic spaces under the perfection functor. The theory of perfect groupoids in algebraic spaces can be found in \cite[\S7]{Liang1}. In the following, we first formalize the notion of characteristics of group algebraic spaces and groupoids in algebraic spaces.
\begin{definition}
Let $B$ be a base algebraic space over $S$. Let $(G,m)$ be a group algebraic space over $B$. Let $(U,R,s,t,c)$ be a groupoid in algebraic spaces over $B$. We say that $(G,m)$ has characteristic $p$ if $G$ has characteristic $p$. We say that $(U,R,s,t,c)$ has characteristic $p$ if $U$ and $R$ have characteristic $p$.
\end{definition}

The following lemma shows that group algebraic spaces (resp. groupoids in algebraic spaces) are stable under the perfection functor.
\begin{lemma}\label{L17}
Let $B$ be a base algebraic space in characteristic $p$ over $S$. If $(G,m)$ is a group algebraic space in characteristic $p$ over $B$, then $(G^{pf},m^{\natural})$ is a group algebraic space over $B^{pf}$. If $(U,R,s,t,c)$ is a groupoid in algebraic spaces in characteristic $p$ over $B$, then $(U^{pf},R^{pf},s^{\natural},t^{\natural},c^{\natural})$ is a groupoid in algebraic spaces over $B^{pf}$.
\end{lemma}
\begin{proof}
This follows from definitions and Proposition \ref{P3}.
\end{proof}

\begin{definition}
Consider the situation as in Lemma \ref{L17}. The perfect group algebraic space $(G^{pf},m^{\natural})$ over $B^{pf}$ is called the perfection of $(G,m)$. The perfect groupoid in algebraic spaces $(U^{pf},R^{pf},s^{\natural},t^{\natural},c^{\natural})$ over $B^{pf}$ is called the perfection of $(U,R,s,t,c)$.
\end{definition}

The canonical morphism $(G^{pf},m^{\natural})\rightarrow(G,m)$ given by $G^{pf}\rightarrow G$ is called the \textit{canonical projection} of $(G^{pf},m^{\natural})$. And the canonical morphism $(U^{pf},F^{pf},s^{\natural},t^{\natural},c^{\natural})\rightarrow(U,F,s,t,c)$ given by $U^{pf}\rightarrow U$ and $F^{pf}\rightarrow F$ is called the \textit{canonical projection} of $(U^{pf},F^{pf},s^{\natural},t^{\natural},c^{\natural})$.

The following lemma shows that morphisms of group algebraic spaces are stable under the perfection functor.
\begin{lemma}
Let $B$ be a base algebraic space in characteristic $p$ over $S$. Let $(G,m)$ and $(G',m')$ be group algebraic spaces in characteristic $p$ over $B$. Let $f:(G,m)\rightarrow(G',m')$ be a morphism of group algebraic spaces over $B$. Then
$$
f^{\natural}:(G^{pf},m^{\natural})\rightarrow(G'^{pf},m'^{\natural})
$$
is a morphism of group algebraic spaces over $B^{pf}$.
\end{lemma}
\begin{proof}
This follows from definitions and Proposition \ref{P3}.
\end{proof}

Moreover, we can show that morphisms of groupoids in algebraic spaces are stable under the perfection functor.
\begin{lemma}
Let $B$ be a base algebraic space in characteristic $p$ over $S$. Let $(U,R,s,t,c)$ and $(U',R',s',t',c')$ be groupoids in algebraic spaces of characteristic $p$ over $B$. Let $f:(U,R,s,t,c)\rightarrow(U',R',s',t',c')$ be a morphism of groupoids in algebraic spaces over $B$. Then
$$
f^{\natural}:(U^{pf},R^{pf},s^{\natural},t^{\natural},c^{\natural})\rightarrow(U'^{pf},R'^{pf},s'^{\natural},t'^{\natural},c'^{\natural})
$$
is a morphism of groupoids in algebraic spaces over $B^{pf}$.
\end{lemma}
\begin{proof}
This follows from definitions and Proposition \ref{P3}.
\end{proof}

The following type of groupoids in algebraic spaces is of particular importance to the study of Deligne-Mumford stacks.
\begin{definition}
Let $B$ be a base algebraic space over $S$. Let $(U,R,s,t,c)$ be a groupoid in algebraic spaces over $B$. We say that $(U,R,s,t,c)$ is an \'{e}tale groupoid in algebraic spaces if the morphisms $s,t$ are \'{e}tale.
\end{definition}
\begin{remark}
Our definition of \'{e}tale groupoids in algebraic spaces is due to \cite[Tag04TH]{Stack Project}.
\end{remark}

We observe that \'{e}tale groupoids in algebraic spaces are stable under fibre product in the following case.
\begin{proposition}
Let $(U,R,s,t,c),(U',R',s',t',c'),(U'',R'',s'',t'',c'')$ be \'{e}tale groupoids in algebraic spaces over $B$. Let $(\varphi_{1},\varphi_{2}):(U,R,s,t,c)\rightarrow(U'',R'',s'',t'',c'')$ and $(\psi_{1},\psi_{2}):(U',R',s',t',c')\rightarrow(U'',R'',s'',t'',c'')$ be \'{e}tale morphisms of groupoids in algebraic spaces over $B$, i.e. $\varphi_{1},\varphi_{2},\psi_{1},\psi_{2}$ are \'{e}tale morphisms of algebraic spaces over $B$. Then $(U\times_{U''}U',R\times_{R''}R',s''',t''',c''')$ is an \'{e}tale groupoid in algebraic spaces over $B$, where $s''',t''':R\times_{R''}R'\rightarrow U\times_{U''}U'$ and $c''':(R\times_{R''}R')\times_{s''',(U\times_{U''}U'),t'''}(R\times_{R''}R')\rightarrow R\times_{R''}R'$ are morphisms of algebraic spaces over $B$.
\end{proposition}
\begin{proof}
It follows from the proof of \cite[\S3.2, Theorem 5]{Acosta} that the morphisms $s''',t'''$ are automatically \'{e}tale such that $(U\times_{U''}U',R\times_{R''}R',s''',t''',c''')$ is an \'{e}tale groupoid in algebraic spaces over $B$.
\end{proof}

The following proposition shows that \'{e}tale groupoids in algebraic spaces are stable under the perfection functor.
\begin{proposition}
Let $B$ be a base algebraic space in characteristic $p$ over $S$. If $(U,R,s,t,c)$ is an \'{e}tale groupoid in algebraic spaces in characteristic $p$ over $B$, then $(U^{pf},R^{pf},s^{\natural},t^{\natural},c^{\natural})$ is an \'{e}tale groupoid in algebraic spaces over $B^{pf}$.
\end{proposition}
\begin{proof}
This follows from Lemma \ref{L17} and Proposition \ref{P2}.
\end{proof}

\end{document}